\documentclass[journal]{IEEEtran}
\usepackage{amsmath,amsfonts,amssymb,euscript, graphicx, 
epsfig,
enumerate,float,afterpage, subfigure, ifthen}%

\newtheorem{thm}{Theorem}

\newtheorem{lem}{Lemma}
\newtheorem{defn}{Definition}

\newcommand{\expect}[1]{\mathbb{E}\left\{#1\right\}}
\newcommand{\defequiv}{\mbox{\raisebox{-.3ex}{$\overset{\vartriangle}{=}$}}}

\newcommand{\bv}[1]{{\boldsymbol{#1} }}
\newcommand{\script}[1]{{{\cal{#1} }}}

\newcommand{\feasible}{{\script{F}}}

\begin{document}

\title
  {Delay Analysis for Max Weight Opportunistic Scheduling in Wireless Systems}
\author{Michael J. Neely \\ $\vspace{-.4in}$%
\thanks{Michael J. Neely is with the  Electrical Engineering department at the University
of Southern California, Los Angles, CA (web: http://www-rcf.usc.edu/$\sim$mjneely).} 
\thanks{This material was presented in part at the Allerton Conference on Communication, 
Control, and Computing, Monticello, IL, Sept. 2008 \cite{neely-maxweight-delay-allerton}.} 
\thanks{This material is supported in part  by one or more of 
the following: the DARPA IT-MANET program
grant W911NF-07-0028, the NSF grant OCE 0520324, 
the NSF Career grant CCF-0747525.}}

\markboth{}{Neely}

\maketitle

\begin{abstract} 
We consider the delay properties of max-weight opportunistic 
scheduling in a multi-user
ON/OFF wireless system, such as a multi-user downlink or uplink.  
It is well known that max-weight scheduling stabilizes the network (and hence
yields maximum throughput) whenever input rates are inside the network
capacity region.  We show that when arrival and channel processes are
independent, average delay of the max-weight policy is 
order-optimal, in the sense that it does not grow with the number
of network links.   While recent queue-grouping algorithms are known
to also yield order-optimal delay, this is the first such result for the simpler
class of max-weight policies.  We then consider multi-rate
transmission models and show that average delay in this case typically 
does increase with the network size due to queues containing a small 
number of  ``residual'' packets.
\end{abstract} 

\section{Introduction} 

\nocite{neely-maxweight-delay-allerton}

We consider the delay properties of max-weight opportunistic scheduling in
a multi-user wireless system.  Specifically, we consider a system with $N$ transmission
links. Each link
receives independent  data that arrives randomly and 
must be queued for eventual transmission.  Separate queues are maintained
by each link $i \in \{1, \ldots, N\}$, so that data arriving to queue $i$ must be transmitted
over link $i$.    The system works in slotted time with normalized
slots $t \in \{0, 1, 2, \ldots\}$.   The channel states of
each link vary randomly from slot to slot, and every slot $t$ the network controller observes the 
current queue backlogs and the current channel states, and selects a single link for
wireless transmission.

This is a classic \emph{opportunistic scheduling} scenario, where the network scheduler
can exploit knowledge of the current state of the time varying channels.   
It is well known that max-weight scheduling policies are throughput optimal in 
such systems, in the sense that they provably stabilize all queues whenever the 
input rate vector is inside the network capacity region.  This stability result 
was first shown by Tassiulas and Ephremides in \cite{tass-server-allocation}
for the special case of ON/OFF channels, and was later generalized to multi-rate 
transmission models and systems with power allocation \cite{kahale} \cite{andrews-downlink} 
\cite{neely-downlink-ton} \cite{now}.  However, the delay properties 
of max-weight scheduling are less understood.  An average delay bound that is linear
in $N$ is derived in \cite{neely-downlink-ton} \cite{now}.  
While this bound is tight in the case of 
correlated arrival and channel processes, it is widely believed to be loose 
for independent arrivals and channels.
 
In this paper, we focus on the special case of 
ON/OFF channels, and show that the max-weight policy 
indeed yields average delay that is $O(1)$ under independence
assumptions.  Thus, average delay does not grow with the network 
size and hence is \emph{order optimal}.   While our previous queue grouping 
results in \cite{neely-order-optimal-ton}
also demonstrate that $O(1)$ delay is possible, this is the first such result for the 
simpler class of max-weight policies. 
Specifically, we first show that for any input rate vector that is within a $\rho$-scaled
version of the capacity region (where $\rho$ represents the network loading 
and satisfies $0 < \rho < 1$), the max-weight rule yields average delay that
is less than or equal to  $\frac{c\log(1/(1-\rho))}{(1-\rho)^2}$, where $c$ is a constant that
does not depend on $\rho$ or $N$.\footnote{The value $c$ is used here to easily
express a delay scaling relationship, and represents 
a generic coefficient that does not depend on $\rho$ or $N$. The value $c$ is not 
necessarily the same in all places it is used.}   This is in comparison to the previous
delay bound of $\frac{cN}{1-\rho}$ developed for max-weight scheduling 
\cite{neely-downlink-ton} \cite{now}.  Note that our 
new bound does not grow with $N$, but
has a worse asymptotic in $\rho$.   We next present a different
analysis that  improves the delay bound to $\frac{c\log(1/(1-\rho))}{1-\rho}$ for systems
with ``$f$-balanced'' traffic rates (to be made precise in later sections).  That is, if arrival
rates are heterogeneous but are more balanced (so that the difference between the maximum
arrival rate and the average arrival rate is sufficiently small), then order-optimal average
delay is maintained while the delay asymptotic in $\rho$ is
improved.

Finally, we consider  systems with 
multi-rate capabilities.  We first provide a delay bound that grows linearly 
with $N$, similar to the bounds in \cite{neely-downlink-ton} \cite{now} but with an improved
coefficient.  We then provide an example multi-rate system and show 
that its average congestion and delay must grow at least linearly with $N$ under \emph{any} scheduling
algorithm, due to many queues having a small number of ``residual'' packets. This is an important
example and demonstrates that the $\Theta(N)$ behavior of the multi-rate delay bound is fundamental and 
cannot be avoided, highlighting a significant difference between single-rate and multi-rate 
systems.

It is known that order-optimal delay requires queue-based scheduling.  Indeed, it
is shown in \cite{neely-order-optimal-ton} that average delay in an $N$-user 
downlink with time varying channels grows at least linearly with $N$ 
if queue-independent algorithms
are used
(such as round-robin or randomized schedulers).  Related results are shown for 
$N\times N$ packet switches in \cite{neely-log-delay-ton}, where 
a delay gap between queue-aware and queue-independent algorithms is developed. 
Delay optimal control laws for multi-user
wireless systems are mostly limited to systems with special symmetric 
structure \cite{tass-server-allocation} \cite{edmund-thesis} \cite{anand-delay-it}.   
Delay optimality results  are developed in \cite{sanjay-heavy-traffic}  for a heavy traffic
regime in the limit as the system loading $\rho$ approaches $1$.   Recent results on 
exponents of the tail of delay distributions are provided in \cite{lin-exponent-allerton} \cite{stolyar-exponent-allerton}, and order-optimal delay for greedy maximal 
scheduling with $\rho$ a constant factor away
from $1$ is considered in \cite{fast-matching-sanjay} \cite{neely-maximal-bursty}.

The max-weight rule is also called the \emph{Longest Connected Queue} (LCQ) 
scheduling rule in the special case of an ON/OFF downlink.  This policy
was developed by Tassiulas and Ephremides in \cite{tass-server-allocation},
where it was shown to support  the full network capacity region and to also 
be  \emph{delay optimal} in the special symmetric case when all 
arrival rates and ON/OFF probabilities are the same for each link. 
The fact that the actual average delay of LCQ in such symmetric cases is 
$O(1)$ was recently proven in \cite{anand-delay-it} (which shows that doubling
the size of a symmetric system does not increase the average delay) and
\cite{neely-order-optimal-ton} (which uses a queue-grouped Lyapunov function
to bound the average delay).  Delay properties of variations of 
LCQ for symmetric Poisson 
systems are considered in \cite{murat-onoff} in the limit of asymptotically large $N$. 
For asymmetric systems, it is shown in \cite{neely-order-optimal-ton} 
that a  \emph{different algorithm}, called the \emph{Largest Connected
Group} (LCG) algorithm, yields $O(1)$ average delay.  However,
the LCG algorithm requires some statistical knowledge to set up a queue-group
structure.  Hence, it is important  to understand the delay properties 
of the simpler 
max-weight rule, which does not require statistical knowledge. 
In this paper, we combine 
the \emph{queue grouping} concepts developed
in \cite{neely-order-optimal-ton} 
together with two novel Lyapunov functions to provide an order-optimal delay
analysis of max-weight.  The first Lyapunov function we use has a weighted 
sum of two different component functions, and  is inspired by
work in \cite{greedy-wu-srikant-j} where a  Lyapunov
function with a similar structure is used in a different context.

In the next section, we specify the network model and review basic concepts
concerning the network capacity region.  Section \ref{section:onoff1} proves
our first delay result for 
the ON/OFF channel model with general heterogeneous traffic rates inside the capacity region. 
Section \ref{section:onoff2} provides our second bound (with a tighter asymptotic
in $\rho$) for the case of heterogeneous traffic rates but under an  $f$-balanced
traffic assumption.  Section \ref{section:multi-rate} treats multi-rate systems. Section 
\ref{section:sim} provides simulation results.

\section{System Model} \label{section:model} 

Consider a multi-user wireless system with $N$ transmission links.  The system  operates
in slotted time with normalized slots $t \in \{0, 1, 2, \ldots\}$.  We assume that data 
is measured in units of fixed size packets, and let $A_i(t)$ represent the number of packets that arrive to link $i\in \{1, \ldots, N\}$ during
slot $t$.   Each link $i$ maintains
a separate queue to store this arriving data, and we let $Q_i(t)$ represent the number
of packets waiting for transmission over link $i$.

Let $S_i(t)$ represent the \emph{channel state} 
for the $i$th channel during slot $t$.  We assume that $S_i(t)$ is a non-negative integer
that represents the current transmission 
rate (in units of packets/slot) 
available over channel $i$ if this channel is selected for transmission on 
slot $t$.  For most of this paper, we consider the simple case of ON/OFF channels, where
$S_i(t) \in \{0, 1\}$ for all channels $i \in \{1, \ldots, N\}$ (multi-rate systems are treated in 
Section \ref{section:multi-rate}). 
 Define $\bv{S}(t) = (S_1(t), \ldots, S_N(t))$ as the channel
state vector. 

Let $\mu_i(t)$ represent the control decision variable on slot $t$, given as follows: 
\[ \mu_i(t) =  \left\{ \begin{array}{ll}
                          S_i(t)  , & \mbox{  if channel $i$ is selected on slot $t$} \\
                             0  ,  & \mbox{  otherwise} 
                            \end{array}
                                 \right. \]
Define  $\bv{\mu}(t) = (\mu_1(t), \ldots, \mu_N(t))$ as the vector of transmission decisions. 
We also call this the \emph{transmission rate vector}, as it determines the instantaneous
transmission rates over each link (in units of packets/slot), where the rate is either $0$ or $1$.
The constraint that at most one channel is selected per slot translates into the 
constraint that $\bv{\mu}(t)$ has at most one non-zero entry (and any non-zero entry $i$
is equal to $S_i(t)$).  Define  $\feasible(t)$ as the set of all such control vectors $\bv{\mu}(t)$ 
that are possible for slot $t$, called the \emph{feasibility set} for slot $t$. 
The queue dynamics for each queue $i\in \{1, \ldots, N\}$ 
are given as follows: 
\begin{equation} \label{eq:dynamics} 
 Q_i(t+1) = \max[Q_i(t) - \mu_i(t), 0] + A_i(t) 
\end{equation} 
subject to the constraint $\bv{\mu}(t) \in \feasible(t)$ for all $t$.

\subsection{Traffic and Channel Assumptions} 

We assume the arrival  processes $A_i(t)$ 
are independent for all $i \in \{1, \ldots, N\}$.  Further, 
each process $A_i(t)$ is  i.i.d. over slots with mean $\lambda_i = \expect{A_i(t)}$ and 
with a finite second moment $\expect{A_i(t)^2} < \infty$.  Similarly, we assume
channel processes $S_i(t)$ are independent of each other and i.i.d. over slots
with probabilities $Pr[S_i(t) = 1] = p_i$ for $i \in \{1, \ldots, N\}$.

\subsection{The Network Capacity Region} 

Suppose the network control policy chooses a transmission rate vector
every slot 
according to a well defined probability law, so that the queue states
evolve according to (\ref{eq:dynamics}).

\begin{defn}  A queue $Q_i(t)$ is \emph{strongly stable} if: 
\[ \limsup_{t\rightarrow\infty} \frac{1}{t} \sum_{\tau=0}^{t-1} \expect{Q_i(\tau)} < \infty \] 
\end{defn}
We say that the network of queues is \emph{strongly stable} if all individual queues
are strongly stable.  Throughout, we shall use the term ``stability'' to refer to strong stability.

Define $\Lambda$ as the \emph{network capacity region}, consisting of the closure of all
arrival rate vectors $\bv{\lambda} = (\lambda_1, \ldots, \lambda_N)$ for which there 
exists a stabilizing control algorithm.  In \cite{tass-server-allocation} it is shown that the 
capacity region $\Lambda$ is the set of all rate vectors
$\bv{\lambda} = (\lambda_1, \ldots, \lambda_N)$ such that for each 
of the $2^N-1$ non-empty link subsets $\script{L} \subset\{1, \ldots, N\}$, we have: 
\begin{equation} \label{eq:onoff-capacity} 
 \sum_{i \in \script{L}} \lambda_i \leq 1 - \Pi_{i \in \script{L}} (1-p_i) 
 \end{equation} 
This is an explicit description of the capacity region $\Lambda$.  The 
following alternative implicit characterization is also useful for analysis 
 (see \cite{now} and references therein): 
\begin{thm} \label{thm:capacity} (Capacity Region $\Lambda$)
The capacity region $\Lambda$ is equal to the set of all (non-negative) 
rate vectors
$\bv{\lambda} = (\lambda_1, \ldots, \lambda_N)$ for which there exists 
a stationary randomized control policy that observes the current channel state
vector $\bv{S}(t)$ and chooses a feasible transmission rate vector $\bv{\mu}(t) \in \feasible(t)$
as a random function of $\bv{S}(t)$, such that: 
\begin{equation} \label{eq:capacity-expectation} 
 \lambda_i =  \expect{\mu_i(t)}  \: \: \mbox{ for all $i \in \{1, \ldots, N\}$} 
 \end{equation}
 where the expectation is taken with respect to the random channel vector $\bv{S}(t)$
 and the  potentially random control action that depends on $\bv{S}(t)$.  $\Box$
\end{thm}

It is easy to see that any non-negative rate vector that is entrywise less than or equal
to a vector $\bv{\lambda} \in \Lambda$ is also contained in $\Lambda$. This follows 
immediately from Theorem \ref{thm:capacity} by modifying the stationary randomized
policy $\bv{\mu}(t)$ that yields $\expect{\mu_i(t)} = \lambda_i$ to a new policy $\hat{\bv{\mu}}(t)$ 
by probabilistically setting each $\mu_i(t)$ value 
to zero with an appropriate probability $q_i$, yielding $\expect{\hat{\mu}_i(t)} 
= \expect{\mu_i(t)} (1-q_i) \leq \expect{\mu_i(t)}$.  

It is also easy to show that the capacity 
region $\Lambda$ is convex and compact (i.e., closed and bounded).  Further, 
if $\expect{S_i(t)} > 0$ for all $i \in \{1, \ldots, N\}$,  
then $\Lambda$ has full dimension of size $N$
and hence has a non-empty interior.

\subsection{The Max-Weight Scheduling Policy} 

Given a rate vector  $\bv{\lambda}$ interior to the capacity region $\Lambda$, a stationary, 
randomized,  queue-independent policy could in principle be designed to stabilize the system, although this 
would require full knowledge of the traffic rates and channel state probabilities.  
However, it is well known that the following queue-aware
\emph{max-weight} policy stabilizes the system whenever the rate vector is interior to 
$\Lambda$, without requiring knowledge of the traffic rates or channel statistics 
\cite{tass-server-allocation}:
Each slot $t$, observe current queue backlogs and channel states $Q_i(t)$ and $S_i(t)$
for each link $i$, and choose to serve the link  $i^*(t)  \in \{1, \ldots, N\}$ 
with the largest $Q_i(t)S_i(t)$ product.  This is also called the \emph{Longest Connected
Queue} policy (LCQ) \cite{tass-server-allocation}, as it serves the queue with the 
largest backlog among all that are currently ON.

The max-weight policy is very important because of its simplicity and its general
stability properties.   However, a tight delay analysis is quite challenging, 
and prior work provides only a loose upper bound on average delay that is $O(N)$, i.e., 
linear in the network size \cite{neely-downlink-ton} \cite{now}. It 
is shown in \cite{neely-order-optimal-ton} that $O(1)$ average delay is possible 
when  both channels and 
packet arrivals are independent across users and across timeslots, and when no
traffic rate is larger than the average traffic rate by more than a specified amount.  
The $O(1)$ 
delay analysis of  \cite{neely-order-optimal-ton} 
uses an algorithm called \emph{Largest Connected
Group} that is different from the max-weight policy and that requires
more statistical knowledge to implement.  In 
the following, we use the queue grouping analysis techniques
of \cite{neely-order-optimal-ton} to show that the simpler  
max-weight policy can \emph{also} provide $O(1)$ average delay, and does so 
for all traffic rates within a $\rho$-scaled version of the capacity region.  However, 
the scaling in $\rho$ is worse than that in \cite{neely-order-optimal-ton}.  Section
\ref{section:onoff2} recovers the same $\rho$ scaling as \cite{neely-order-optimal-ton}
under a similar ``$f$-balanced'' traffic  assumption.  

\section{Delay Analysis for Arbitrary Rates in $\Lambda$} \label{section:onoff1}

Consider the ON/OFF channel model where each 
$S_i(t)$ is an independent i.i.d. Bernoulli process with $Pr[S_i(t) = 1] = p_i$.
Assume the arrival rate vector $\bv{\lambda} = (\lambda_1, \ldots, \lambda_N)$
is interior to the capacity region $\Lambda$, so that there exists a value $\rho$ 
such that $0 < \rho < 1$ and: 
\begin{equation} 
\bv{\lambda} \in \rho\Lambda \label{eq:rho-lambda} 
\end{equation}
That is, $\bv{\lambda}$ is contained within a $\rho$-scaled version of the capacity region.
The parameter $\rho$ can be viewed as the \emph{network loading}, measuring the fraction 
the rate vector $\bv{\lambda}$ is away from the capacity region boundary.  Define $A_{tot}(t)$ 
as the total packet arrivals on slot $t$: 
\[ A_{tot}(t) \defequiv \sum_{i=1}^N A_i(t) \]
Define $\lambda_{tot}  = \sum_{i=1}^N \lambda_i$ as the sum packet arrival rate. Because
the sum of the entries of any rate vector in the capacity region $\Lambda$ is no 
more than $1$, we have by (\ref{eq:rho-lambda}) 
that $\lambda_{tot} \leq \rho$.

\subsection{Important Parameters of $\Lambda$}

To analyze delay, it is useful to characterize the $N$-dimensional 
capacity region $\Lambda$ in terms
of its size on subspaces of smaller dimension.
To this end, define $p_{min}$ as the smallest channel ON probability: 
\[ p_{min} \defequiv \min_{i\in\{1, \ldots, N\}} p_i \]
We assume that $0  < p_{min} < 1$. 
For each positive  integer $K$, define
parameters $\mu_K^{sym}$ and $r_K$ as follows: 
\begin{eqnarray*} 
\mu_K^{sym} &\defequiv& \frac{1}{K}[1-(1-p_{min})^K] \\
r_K &\defequiv& 1 - (1-p_{min})^K
\end{eqnarray*}
Thus, $r_K = K\mu_K^{sym}$. The following lemma shall be useful. 
\begin{lem} \label{lem:mu-sym}  For any positive integer $K$ and any probability $p_{min}>0$, 
we have $\mu_{K}^{sym} > \mu_{K+1}^{sym}$. 
That is: 
\[ \frac{1}{K}[1 - (1-p_{min})^K] > \frac{1}{K+1}[1 - (1-p_{min})^{K+1}] \]
\end{lem}
\begin{proof}
See Appendix D. 
\end{proof} 

 Further, for $K \in \{1, \ldots, N\}$, 
let  $\script{L}_K$ represent a particular subset of $K$ links within the link set $\{1, \ldots, N\}$.
For each subset $\script{L}_K$, define $\bv{1}_{\script{L}_K}$ as an $N$-dimensional 
vector that is $1$ in all entries $i \in \script{L}_K$, and zero in all other entries.  

\begin{lem} \label{lem:symmetry-fact} For each set $\script{L}_K$ of size $K$ (for any integer 
$K$ such that $1 \leq K \leq N$) we have: 
\[ \mu_{K}^{sym}\bv{1}_{\script{L}_K} \in \Lambda  \]
Furthermore, for each integer $k$ such that $1 \leq k \leq K$ and for any set $\script{L}_k$
that contains $k$ links,  we have 
$\mu_{K}^{sym} \bv{1}_{\script{L}_k} \in \Lambda$. 
\end{lem}
\begin{proof}
We first prove that $\mu_K^{sym} \bv{1}_{\script{L}_K} \in \Lambda$. 
By (\ref{eq:onoff-capacity}), it suffices to show that for any integer $m$ such that $1 \leq m \leq K$,  
the sum of any $m$ non-zero 
components of $\mu_{K}^{sym}\bv{1}_{\script{L}_K}$ is less than or equal to $r_m$.\footnote{Note
that $r_m \leq 1 - \Pi_{i \in \script{L}_m}(1-p_i)$, where $\script{L}_m$ is any subset
of $m$ links.}
That is, it suffices to show that $m\mu_K^{sym} \leq r_m$. 
But this is equivalent to showing that $\mu_K^{sym} \leq \mu_m^{sym}$ for 
$m \leq K$, which is true by Lemma \ref{lem:mu-sym}.   Finally, the
fact that $\mu_K^{sym}\bv{1}_{\script{L}_k}\in\Lambda$ (for any integer 
$k$ such that $1 \leq k \leq K$)
follows because any 
rate vector with entries less than or equal to another rate vector in $\Lambda$ is also in $\Lambda$.
\end{proof}

Thus, $\mu_K^{sym}$ can be intuitively viewed as an edge size such that any 
$K$-dimensional hypercube of this edge size (with dimensions defined along the orthogonal 
directions of any $K$ axes of
$\mathbb{R}^N$) can fit inside the capacity region 
$\Lambda$.

\subsection{The $O(1)$ Delay Bound for Arbitrary Traffic in $\Lambda$} 

Suppose the LCQ algorithm is used together with a stationary probabilistic tie breaking
rule in cases when multiple queues have the same weight. 
 This allows the queueing system to be  viewed as a 
stationary Markov chain. In this case, it is well known that if the arrival rate vector is 
interior to the capacity region, then all queues are stable under LCQ, with a well defined steady
state time average \cite{now}.   The following $O(N)$ delay bound  for LCQ 
is given in \cite{neely-order-optimal-ton}:\footnote{The 
bound in \cite{neely-order-optimal-ton} is of the form $c/\epsilon$, where 
$\epsilon$ is any value such that $\bv{\lambda} + \bv{\epsilon} \in \Lambda$, where
$\bv{\epsilon}$ is a vector with all values equal to $\epsilon$.  The bound (\ref{eq:old-onoff})
follows by observing that $\epsilon = (1-\rho)r_N/N$ satisfies $\bv{\lambda} + \bv{\epsilon} \in \Lambda$
whenever $\bv{\lambda} \in \rho \Lambda$. A similar $c/\epsilon$ bound is given in 
\cite{neely-downlink-ton} \cite{now} for more general multi-rate systems.}

\begin{equation} \label{eq:old-onoff} 
\overline{W} \leq \frac{N[1 + \frac{1}{\lambda_{tot}}\sum_{i=1}^N\expect{A_i(t)^2} - \frac{2}{\lambda_{tot}}\sum_{i=1}^N \lambda_i^2]}{2r_N(1-\rho)} 
\end{equation}
where $\overline{W}$ represents the average delay in the system. 
The bound (\ref{eq:old-onoff}) also holds for arrival vectors  $\bv{A}(t)$ that are i.i.d. over slots but with 
possibly correlated entries $A_i(t)$ on the same slot $t$.
The next theorem demonstrates an improved $O(1)$ bound in the case when all arrival
processes $A_i(t)$ are independent.

\begin{thm} \label{thm:onoff1}  (Delay Bound for LCQ) Consider the ON/OFF channel model and  
assume processes $A_i(t)$ and $S_i(t)$ are independent and
i.i.d. over slots. Assume that $\bv{\lambda} \in \rho\Lambda$ 
for some network loading $\rho$ such that $0 < \rho < 1$. 
 Let $K$ be any integer such that $r_{K+1} > \lambda_{tot}$, that is:\footnote{Note that  
 $\lambda_{tot} \leq \rho < 1$, and hence there is 
 always a suitably large value $K$ such that (\ref{eq:k-condition}) holds.}  
\begin{equation} \label{eq:k-condition} 
 1 - (1-p_{min})^{K+1}  > \lambda_{tot} 
 \end{equation} 
Then the max-weight (LCQ) 
policy for the ON/OFF channel model  stabilizes all queues and yields: 
\begin{equation*}  
\sum_{i=1}^N \overline{Q}_i \leq \frac{KB_{\theta}C}{(1-\rho)^2}
\end{equation*} 
where $\overline{Q}_i$ is the time average number of packets in queue $i$, and where the 
constants $B_{\theta}$,  $C$, and $\theta$ are defined: 
\begin{eqnarray} 
B_{\theta} &\defequiv& \frac{\lambda_{tot}}{2} + \frac{1}{2}\sum_{i=1}^N\expect{A_i(t)^2} - \sum_{i=1}^N\lambda_i^2 \nonumber \\
&& + \frac{\theta}{2} \left[\expect{A_{tot}(t)^2} + \lambda_{tot} - 2\lambda_{tot}^2  \right] \label{eq:B} \\
C &\defequiv&  \left\{ \begin{array}{ll}
                           \frac{r_{K+1}}{\frac{r_NK\lambda_{tot}}{N(1-\rho)}
+ \frac{r_K(r_{K+1} - \lambda_{tot})}{(1-\rho)}} &\mbox{ if $K < N$} \\
                             1/r_N  & \mbox{ if $K \geq N$} 
                            \end{array}
                                 \right. \\
\theta &\defequiv&   \left\{ \begin{array}{ll}
                        \frac{(1-\rho)(\mu_K^{sym} - \mu_N^{sym})}{r_{K+1}}   &\mbox{ if $K < N$} \\
                             0  & \mbox{ if $K \geq N$} 
                            \end{array}
                                 \right. \\
\end{eqnarray} 
 By Little's Theorem, average
delay $\overline{W}$ thus satisfies: 
\begin{equation} \label{eq:lcq1} 
\overline{W} \leq \min\left[\frac{KB_{\theta}C}{\lambda_{tot}(1-\rho)^2} , \frac{NB_0}{\lambda_{tot}r_N(1-\rho)}  \right]
\end{equation} 
where $B_0$ represents the value of $B_{\theta}$ with $\theta = 0$, and the second expression in the 
above $\min[\cdot, \cdot]$ function is identical to the previous delay bound (\ref{eq:old-onoff}). 
\end{thm}

The proof of Theorem \ref{thm:onoff1} is given in the next subsection.  
We note that the right term inside the  $\min[\cdot, \cdot]$ operator in (\ref{eq:lcq1}) 
 is smaller in the case $K \geq N$.  The above bound can be minimized over all positive integers $K$ that 
satisfy $r_{K+1} >  \lambda_{tot}$. 
For a simpler interpretation of the bound that illuminates the fact that this is an $O(1)$ delay
result, note that because $(1+\rho)/2 > \rho \geq \lambda_{tot}$, 
we can ensure that (\ref{eq:k-condition}) holds
by choosing $K$ to satisfy: 
\[ 1 - (1-p_{min})^{K+1} \geq (1+\rho)/2 \]
Choosing $K$ as follows accomplishes this:
\begin{equation} \label{eq:K-example} 
 K  = \max\left[1, \left\lceil  \frac{\log(2/(1-\rho))}{\log(1/(1-p_{min}))}   \right\rceil - 1\right] 
 \end{equation} 
Because $\lambda_{tot} \leq \rho$, it 
is not difficult to show that with this choice of $K$, we have $r_{K+1} - \lambda_{tot} \geq 
(1-\rho)/2$.  Thus, in the case $K < N$ we have: 
\[ C \leq \frac{r_{K+1}}{\frac{r_K(r_{K+1}- \lambda_{tot})}{(1-\rho)}} \leq \frac{2r_{K+1}}{r_{K}} \]
Because $C = 1/r_N$ for the case $K \geq N$, we have that $C = O(1)$ (regardless of whether or 
not $K<N$). 
Further, we have from (\ref{eq:K-example}) that 
$K$ is proportional to $\log(1/(1-\rho))$ but independent of $N$.  Finally, if arrival processes are 
independent so that $\expect{A_{tot}(t)^2} = O(1)$, we have $B_{\theta}/\lambda_{tot} = O(1)$. Therefore, the 
delay bound of (\ref{eq:lcq1}) has the form: 
\begin{equation} 
 \overline{W} \leq \min\left[ \frac{c_1 \log(1/(1-\rho))}{(1-\rho)^2}, \frac{c_2 N}{(1-\rho)}\right] 
 \end{equation} 
where $c_1$ and $c_2$ are constants that do not depend on $\rho$ or $N$.  If $N$ itself is small, 
then the right expression in the above $\min[\cdot, \cdot]$ can be smaller than the left expression (i.e., 
the previous delay bound (\ref{eq:old-onoff}) can be the same as our new delay bound in the case when
$N$ is small). 
However, if the loading $\rho$ is held fixed as $N$ is scaled to infinity, then the left expression in the 
$\min[\cdot, \cdot]$ is always smaller and demonstrates $O(1)$ average 
delay (see also simulations in Figs. \ref{fig:1} and \ref{fig:2} of Section \ref{section:sim}). 
Thus, LCQ is \emph{order-optimal}
with respect to $N$.  However, the left delay bound has a worse asymptotic in $\rho$, and so it would
be worse than the right bound in the opposite case when $N$ is fixed and $\rho$ is scaled to $1$.

\subsection{Lyapunov Drift Analysis} 

To prove Theorem \ref{thm:onoff1}, it suffices to consider only the case $K < N$, as the delay bound in the 
opposite case $K \geq N$ is identical to the previous delay bound (\ref{eq:old-onoff}).
Let $\bv{Q}(t) = (Q_1(t), \ldots, Q_N(t))$ be the vector of queue backlogs. 
Define $Q_{tot}(t)$ as the sum queue backlog in all queues of the system: 
\begin{equation} \label{eq:q-tot}
Q_{tot}(t) \defequiv \sum_{i=1}^N Q_i(t)  
\end{equation} 
Define the following Lyapunov function:  
\begin{eqnarray}
&L(\bv{Q}(t)) \defequiv \frac{1}{2} \sum_{i=1}^N Q_i(t)^2  + \frac{\theta}{2}\left(\sum_{j=1}^N Q_j(t)\right)^2&\label{eq:lyapunov-function1} 
\end{eqnarray} 
where $\theta$ is a positive constant to be determined later. 
Thus, $L(\bv{Q}(t)) = \frac{1}{2} \sum_{i=1}^N Q_i(t)^2 + \frac{\theta}{2}Q_{tot}(t)^2$. 
This Lyapunov function uses the standard sum of squares of queue length,  and 
adds a new term that is the square of the total queue backlog.
This new term incorporates the \emph{queue grouping} concept similar 
to \cite{neely-order-optimal-ton}, and will be important in obtaining tight delay bounds. 
The technique of composing this Lyapunov function as a sum of two different 
quadratic terms weighted by a $\theta$ constant shall be useful in analyzing both stability 
and delay in  two different
modes of network operation, and is inspired by a 
similar technique  
 used in \cite{greedy-wu-srikant-j} to analyze stability in a very 
 different context.  
 Specifically, work in \cite{greedy-wu-srikant-j} considers
 multi-hop networks with greedy maximal scheduling and achieves stability results
 when input rates are a constant factor (such as a factor of 2) 
 away from the capacity region boundary. 
 
 Here, we consider a single-hop network with time-varying channels, and obtain both
 stability and order-optimal delay results for all input rates inside the capacity region.
 The intuition on why this 2-part Lyapunov function allows a tight delay bound is 
 as follows:  The first term is a standard sum of squares of queue length, and ensures
 stability of the algorithm while creating a large negative drift when the number of non-empty
 queues is small.  However, this term also has a 
 relatively small negative drift when the number of non-empty queues is
 large, preventing $O(1)$ delay analysis from this term alone.  
 To compensate, the second term is a square of the sum of all queues, which does
 not significantly affect the drift of the first term when the number of non-empty queues
 is small, but creates a large negative drift to \emph{help} the first term 
 when the number of non-empty queues is large.

 The queue dynamics (\ref{eq:dynamics}) can be rewritten as follows:
\begin{equation}  \label{eq:dynamics-tilde}  
  Q_i(t+1) = Q_i(t) - \tilde{\mu}_i(t) + A_i(t) 
  \end{equation} 
 where $\tilde{\mu}_i(t) = \min[Q_i(t), \mu_i(t)]$.  Define $\tilde{\mu}_{tot}(t) = \sum_{i=1}^N \tilde{\mu}_i(t)$, being either $0$ or $1$, and being $1$ if and only if 
 the system serves a packet on slot $t$. 
 The dynamics for $Q_{tot}(t)$ are given by: 
 \begin{equation} \label{eq:dynamics-qtot} 
  Q_{tot}(t+1)  = Q_{tot}(t) -\tilde{\mu}_{tot}(t) + A_{tot}(t) 
  \end{equation} 
 where $A_{tot}(t) = \sum_{i=1}^N A_i(t)$. 
 Let $\bv{Q}(t)$ be the stochastic queue evolution process for  a given control policy. 
Define the one-step conditional Lyapunov drift as follows:\footnote{Strictly speaking, 
correct notation should be $\Delta(\bv{Q}(t), t)$, as the drift could be from a non-stationary
policy, although we use the simpler notation $\Delta(\bv{Q}(t))$ 
as formal notation for the right hand side of (\ref{eq:drift-def}).}
\begin{equation} \label{eq:drift-def}
 \Delta(\bv{Q}(t)) \defequiv \expect{L(\bv{Q}(t+1)) - L(\bv{Q}(t))\left|\right.\bv{Q}(t)} 
 \end{equation}

 \begin{lem} \label{lem:drift1}  The Lyapunov drift $\Delta(\bv{Q}(t))$ for the ON/OFF channel model 
 satisfies: 
 \begin{eqnarray*}
 \Delta(\bv{Q}(t)) &=& \expect{B(t)\left|\right.\bv{Q}(t)} \\
 && -  \sum_{i=1}^N Q_i(t)\expect{\mu_i(t) - \lambda_i\left|\right.\bv{Q}(t)} \\
 && - \theta Q_{tot}(t) \expect{\tilde{\mu}_{tot}(t) - \lambda_{tot}\left|\right.\bv{Q}(t)} 
 \end{eqnarray*}
 where $\mu_i(t)$ and $\tilde{\mu}_{tot}(t)$ correspond to the LCQ policy, and 
 where $B(t)$ is given by: 
 \begin{eqnarray}
 \hspace{-.2in}B(t) &\defequiv& \frac{\tilde{\mu}_{tot}(t)}{2} + \frac{1}{2} \sum_{i=1}^N [A_i(t)^2 
 - 2A_i(t)\tilde{\mu}_i(t)]  \nonumber \\
 && + \frac{\theta}{2}[A_{tot}(t)^2+ \tilde{\mu}_{tot}(t) - 2\tilde{\mu}_{tot}(t)A_{tot}(t)] \label{eq:bt}
 \end{eqnarray}
 \end{lem} 
 
 \begin{proof} (Lemma \ref{lem:drift1}) 
 See Appendix A. 
 \end{proof}

Now note that the LCQ algorithm chooses $\bv{\mu}(t) \in \feasible(t)$ on each slot 
$t$ to maximize
 $\sum_{i=1}^N Q_i(t) \mu_i(t)$, and hence: 
 \[ \sum_{i=1}^N Q_i(t)\mu_i(t) \geq \sum_{i=1}^NQ_i(t)\mu_i^*(t)  \]
 where $\bv{\mu}^*(t) = (\mu_1^*(t), \ldots, \mu_N^*(t))$ is any other feasible
 transmission rate vector in $\feasible(t)$.  It follows that the above inequality is
 preserved when taking conditional expectations given the current
 $\bv{Q}(t)$ value.  Plugging this result into the second 
 term on the right hand side of the drift expression in Lemma \ref{lem:drift1} thus
 yields: 
 \begin{eqnarray}
 \Delta(\bv{Q}(t)) &\leq& \expect{B(t)\left|\right.\bv{Q}(t)} \nonumber \\
 && - \sum_{i=1}^N Q_i(t)\expect{\mu_i^*(t) - \lambda_i\left|\right.\bv{Q}(t)} \nonumber \\
 && - \theta Q_{tot}(t) \expect{\tilde{\mu}_{tot}(t) - \lambda_{tot}\left|\right.\bv{Q}(t)}   \label{eq:drifty}
 \end{eqnarray}
where $\bv{\mu}^*(t) = (\mu_1^*(t), \ldots, \mu_N^*(t))$ is any other feasible control
action on slot $t$. Note that $\tilde{\mu}_{tot}(t)$ in the above expression still corresponds
to the LCQ policy. 

Let $L(t)$ represent the number of non-empty queues on slot $t$, so that 
$0 \leq L(t) \leq N$. 

\begin{itemize} 
\item \emph{\underline{Case 1}} ($L(t) \leq K$):  Suppose $L(t) \leq K$, and 
let $\script{L}(t)$ represent the set of non-empty queue
indices. Recall that $\mu_K^{sym} \bv{1}_{\script{L}(t)} \in \Lambda$ (by Lemma \ref{lem:symmetry-fact})
and that $\bv{\lambda}/\rho \in 
\Lambda$ (by assumption that $\bv{\lambda} \in \rho\Lambda$).  
By taking a convex combination of these two vectors and using 
convexity of the set $\Lambda$, it follows that: 
\begin{equation} \label{eq:convexity1} 
 \bv{\lambda} + (1-\rho)\mu_K^{sym} \bv{1}_{\script{L}(t)} \in \Lambda  
 \end{equation} 
Now let $\bv{\mu}^*(t)$ be the stationary randomized policy that makes decisions
based only on the current channel state, and that yields: 
\[ \expect{\bv{\mu}^*(t)} = \bv{\lambda} + (1-\rho)\mu_K^{sym}\bv{1}_{\script{L}(t)} \]
Such a policy exists by (\ref{eq:convexity1}) and 
Theorem \ref{thm:capacity}. Thus, 
for all $i \in \script{L}(t)$ we have:
\begin{equation} \label{eq:stat1} 
\expect{\mu_i^*(t)} = \lambda_i  + (1-\rho)\mu_K^{sym} 
\end{equation}  
Using (\ref{eq:stat1}) in the drift inequality (\ref{eq:drifty}) and noting that $Q_i(t) = 0$ if 
$i \notin \script{L}(t)$ yields:
\begin{eqnarray*}
\Delta(\bv{Q}(t)) &\leq& \expect{B(t)\left|\right.\bv{Q}(t)} - \sum_{i=1}^N Q_i(t)(1-\rho)\mu_{K}^{sym} \\
&& +   \theta Q_{tot}(t)\lambda_{tot}  \\
&=& \expect{B(t)\left|\right.\bv{Q}(t)} \\
&& - Q_{tot}(t)[(1-\rho)\mu_K^{sym} - \theta\lambda_{tot} ]
\end{eqnarray*}
Define $\epsilon$ as follows: 
\begin{equation}\label{eq:epsilon} 
\epsilon \defequiv [(1-\rho)\mu_K^{sym} - \theta\lambda_{tot} ]
\end{equation} 
It follows that: 
\begin{equation} \label{eq:case1-drift} 
\Delta(\bv{Q}(t)) \leq \expect{B(t)\left|\right.\bv{Q}(t)} - \epsilon Q_{tot}(t)
\end{equation} 

\item \emph{\underline{Case 2}} ($L(t) > K$):  Suppose $L(t) >K$, and again let $\script{L}(t)$ represent the set of 
non-empty queue indices.  Note that $\bv{\lambda}/\rho \in \Lambda$ and $\mu_N^{sym}\bv{1} \in \Lambda$, where $\bv{1}$ is the
all 1 vector. By convexity of $\Lambda$, the convex combination is also in $\Lambda$: 
\[ \bv{\lambda} + (1-\rho)\mu_N^{sym}\bv{1} \in \Lambda \]
Now let $\bv{\mu}^*(t)$ be the stationary randomized policy that makes
decisions independent of queue backlog, and that yields for all $i \in \{1, \ldots, N\}$:  
\begin{equation} \label{eq:case2a} 
 \expect{\mu_i^*(t)} = \lambda_i + (1-\rho)\mu_N^{sym} 
 \end{equation} 
Such a policy exists by Theorem \ref{thm:capacity}. 
Note that when the number of non-empty queues is greater than $K$, there is a packet departure
under the LCQ policy with probability at least one minus the product of the $K+1$ largest
OFF probabilities: 
\begin{equation} \label{eq:case2b} 
\expect{\tilde{\mu}_{tot}\left|\right.\bv{Q}(t)} \geq 1 - \Pi_{i \in \hat{\script{L}}_{K+1}}(1-p_i) \geq r_{K+1}
\end{equation}
where $\hat{\script{L}}_{K+1}$ represents the  set of $K+1$ 
links with the smallest success probabilities. Plugging (\ref{eq:case2a}) and (\ref{eq:case2b})
into the drift inequality (\ref{eq:drifty}) yields: 
\begin{eqnarray*}
\Delta(\bv{Q}(t)) &\leq& \expect{B(t)\left|\right.\bv{Q}(t)} \\
&&  \hspace{-.15in} - Q_{tot}(t)[(1-\rho)\mu_N^{sym} + \theta(r_{K+1} - \lambda_{tot}) ]
\end{eqnarray*}
\end{itemize}

To equalize the drift in both Case 1 and Case 2, we choose $\theta$ to satisfy: 
\[ \epsilon  = (1-\rho)\mu_N^{sym} + \theta(r_{K+1} - \lambda_{tot}) \]
Thus (using (\ref{eq:epsilon})): 
\begin{eqnarray*}
\theta &=& \frac{(1-\rho)(\mu_K^{sym} - \mu_N^{sym})}{r_{K+1}} \\
\epsilon &=& \frac{(1-\rho)\left[ \mu_N^{sym}\lambda_{tot} +
 \mu_K^{sym}(r_{K+1}-\lambda_{tot}) \right] }{r_{K+1}}
\end{eqnarray*}

Recall that we have assumed $K<N$ (as Theorem \ref{thm:onoff1} is trivially true if $K \geq N$, as
described at the beginning of this subsection).  Thus, we have $\mu_K^{sym} > \mu_N^{sym}$ (by 
Lemma \ref{lem:mu-sym}), and so we indeed have $\theta >0$. 
Further, because $r_{K+1} >  \lambda_{tot}$,  
we have that $\epsilon>0$.
Therefore, the drift inequality (\ref{eq:case1-drift}) holds in both Case 1 and Case 2 (and
hence holds for all $t$ and all  $\bv{Q}(t)$). 
We now use the following well known 
Lyapunov drift lemma (see, for example, \cite{now} for a proof):

\begin{lem} \label{lem:lyap-drift} (Lyapunov Drift \cite{now}) If the drift $\Delta(\bv{Q}(t))$ of a non-negative
Lyapunov function satisfies the following for all $t$ and all $\bv{Q}(t)$: 
\[ \Delta(\bv{Q}(t)) \leq \expect{B(t)\left|\right.\bv{Q}(t)}  - \epsilon \expect{f(t)\left|\right.\bv{Q}(t)}  \]
for some stochastic processes $B(t)$, $f(t)$, and some constant $\epsilon>0$, then: 
\[ \overline{f} \leq \overline{B}/\epsilon \]
where
\begin{eqnarray*}
\overline{f} &\defequiv& \limsup_{t\rightarrow\infty} \frac{1}{t} \sum_{\tau=0}^{t-1} \expect{f(\tau)} \\
\overline{B} &\defequiv& \limsup_{t\rightarrow\infty} \frac{1}{t}\sum_{\tau=0}^{t-1} \expect{B(\tau)}  \Box
\end{eqnarray*}
\end{lem} 

Using this Lyapunov drift lemma in (\ref{eq:case1-drift}) (using $f(t) = Q_{tot}(t)$) yields: 
\[ \overline{Q}_{tot } \leq \overline{B}/\epsilon \]
We note that because the system evolves according to a Markov chain with a countably infinite
state space, the time averages are well defined (so that the $\limsup$ can be replaced by a regular
limit).  Further, using the fact  that
$\lim_{t\rightarrow\infty} \frac{1}{t} \sum_{\tau=0}^{t-1} \expect{\tilde{\mu}_i(\tau)} = \lambda_i$, the
value of $\overline{B}$ can be seen to equal the value $B_{\theta}$ defined
in (\ref{eq:B}), proving Theorem \ref{thm:onoff1}.

\section{A Tighter Bound  for ``$f$-Balanced Traffic''} \label{section:onoff2}

Here we present a tighter bound on average backlog and delay of the LCQ algorithm
for the ON/OFF channel model.  Our bound in this section 
is of the form $c\log(\frac{1}{1-\rho})/(1-\rho)$, which is still $O(1)$ with respect
to the network size $N$, but yields a better asymptotic in $\rho$.   Unfortunately, 
our analysis does not hold for all rate  vectors $\bv{\lambda}$ inside the capacity
region $\Lambda$.  Rather, we make the following assumption about a more ``balanced''
traffic rate vector.
Let $\bv{\lambda}  = (\lambda_1, \ldots, \lambda_N)$,  and without loss of generality assume
that $\lambda_i > 0$ for all $i \in \{1, \ldots, N\}$ (else, we can redefine $N$ to be the number
of links with non-zero rates). Define $\lambda_{tot} = \sum_{i=1}^N 
\lambda_i$ and $\lambda_{av} = \lambda_{tot}/N$.  
We say that $\bv{\lambda}$ has \emph{$f$-balanced rates}
if there is a constant $f$ such that: 
\begin{equation} \label{eq:order-balanced} 
\lambda_i \leq \lambda_{av} + f \: \: \mbox{ for all $i \in \{1, \ldots, N\}$} 
\end{equation} 
That is, $\bv{\lambda}$ is $f$-balanced if no individual
traffic rate is  more than an amount $f$ above the average rate $\lambda_{av}$. 
Clearly any \emph{uniform} traffic rate vector is $f$-balanced
for $f=0$.  
However, this definition of $f$-balanced rates also captures a 
large class of heterogeneous arrival rate vectors.   We shall prove our delay results under
the assumption that $f$ is suitably small.  A similar assumption is 
used in \cite{neely-order-optimal-ton}, and our delay analysis 
relies heavily on the queue-grouping techniques used there. 

 \subsection{The Queue-Grouped Lyapunov Function} 
 
Fix an integer $K$ such that $1 \leq K \leq N$.  Define 
$\hat{N}$ as the smallest multiple of $K$ that is larger than or equal to $N$: 
\begin{equation} \label{eq:nhat} 
 \hat{N} = \lceil N/K \rceil K 
 \end{equation} 
Now define a new rate vector $\hat{\bv{\lambda}} = (\lambda_{1}, \ldots, \lambda_N, 0, 0, \ldots, 0)$,
where the last $\hat{N} - N$ entries are zero. 
Define $\hat{N} - N$ ``fictitious'' queues for these last dimensions (these queues always have
zero backlog, but shall be convenient to define for counting purposes). 
Define $\script{G}_K$ as the set of all possible partitions of the link set 
$\{1, \ldots, \hat{N}\}$ into  $K$ disjoint sets, each 
with an equal size of  $\hat{N}/K$ links. Let $g \in \script{G}_K$
denote a particular partition, and define $\script{L}_1^{(g)}, \ldots, \script{L}_K^{(g)}$
as the collection of sets corresponding to $g$ (so that the union
$\cup_{k=1}^K \script{L}_{k}^{(g)}$ is equal to $\{1, \ldots, \hat{N}\}$, and the intersection 
$\script{L}_{n}^{(g)} \cap \script{L}_{m}^{(g)}$ is 
empty for all $m \neq n$, where $m, n \in \{1, \ldots, K\}$). 

For a particular partition $g$, define $Q_k^{(g)}(t)$ as the sum of all queue backlogs
in the $k$th set of $g$: 
\[ Q_k^{(g)}(t) \defequiv \sum_{i\in \script{L}_k^{(g)}} Q_i(t) \]
Define the following \emph{queue-grouped Lyapunov function}: 
\begin{equation} \label{eq:queue-group-lyap}
L(\bv{Q}(t)) \defequiv \frac{1}{2}\sum_{g \in \script{G}_K} \sum_{k=1}^K (Q_k^{(g)}(t))^2
\end{equation}  
This is similar to the Lyapunov function of \cite{neely-order-optimal-ton}, with the exception
that it sums over all possible partitions into $K$ disjoint groups.  For intuition, we note that
the $f$-balanced traffic assumption allows the ``Largest Connected Group'' (LCG) argument 
of \cite{neely-order-optimal-ton} to proceed on \emph{any set of $K$ disjoint groups}.  However, 
once we fix  a particular group, minimizing the drift gives rise to the LCG algorithm rather than the 
``max-weight'' LCQ algorithm.   Changing the Lyapunov function by 
summing over all possible $K$ disjoint groups yields a similar negative drift as in LCG, 
but the ``symmetry'' induced by summing over all groups
remarkably makes the drift minimizing algorithm the LCQ algorithm (rather than LCG).

Define $A_k^{(g)}(t)$ and $\tilde{\mu}_k^{(g)}(t)$ as the sum arrivals and departures
from the $k$th group of the partition  $g$: 
\begin{eqnarray*}
A_k^{(g)}(t) &\defequiv& \sum_{i\in\script{L}_k^{(g)}} A_i(t) \\
\tilde{\mu}_k^{(g)}(t) &\defequiv& \sum_{i\in\script{L}_k^{(g)}} \tilde{\mu}_i(t)
\end{eqnarray*}
The dynamics for the $k$th group of partition $g$ 
thus satisfy: 
\begin{equation} \label{eq:group-dynamics} 
Q_k^{(g)}(t+1) = Q_k^{(g)}(t) - \tilde{\mu}_k^{(g)}(t) + A_k^{(g)}(t) 
\end{equation} 
Define the Lyapunov drift $\Delta(\bv{Q}(t))$ as before (given in 
(\ref{eq:drift-def})). 

\begin{lem} \label{lem:drift-group} For a general scheduling policy, 
the Lyapunov drift satisfies: 
\begin{eqnarray*}
\Delta(\bv{Q}(t)) &=& \expect{C(t)\left|\right.\bv{Q}(t)} \\
&& - \sum_{g \in \script{G}_K}\sum_{k=1}^K Q_k^{(g)}(t)\expect{\tilde{\mu}_k^{(g)}(t) - \lambda_k^{(g)}\left|\right.\bv{Q}(t)}
\end{eqnarray*}
where $\lambda_k^{(g)}\defequiv \sum_{i \in \script{L}_k^{(g)}} \lambda_i$, and where
$C(t)$ is defined: 
\begin{eqnarray}
&&\hspace{-.4in} C(t) \defequiv \frac{1}{2}\sum_{g \in \script{G}_K}\sum_{k=1}^K\left[\tilde{\mu}_k^{(g)}(t) + A_k^{(g)}(t)^2 - 2\tilde{\mu}_k^{(g)}(t) A_{k}^{(g)}(t)\right] \nonumber \\
&&\hspace{-.4in} \label{eq:c-t-onoff2}
\end{eqnarray}
\end{lem} 
\begin{proof} 
The proof is similar to the proof of Lemma \ref{lem:drift1}.  Specifically, note from (\ref{eq:group-dynamics})
that: 
\begin{eqnarray*}
&\hspace{-.6in} Q_k^{(g)}(t+1)^2 - Q_k^{(g)}(t)^2 = 
\tilde{\mu}_k^{(g)}(t) + A_k^{(g)}(t)^2 \\
&\hspace{+.3in} - 2\tilde{\mu}_k^{(g)}(t)A_k^{(g)}(t) - 2Q_k^{(g)}(t)[\tilde{\mu}_k^{(g)}(t) - A_k^{(g)}(t)]
\end{eqnarray*}
where we have used the fact that $\tilde{\mu}_k^{(g)}(t)^2 = \tilde{\mu}_k^{(g)}(t)$. 
The result follows by summing over all $k$ and all groups, and taking conditional expectations. 
\end{proof} 

Remarkably, we next show that the ``max-weight'' LCQ algorithm for this ON/OFF channel model
minimizes the final term in the right hand side of the above drift expression.  

\begin{lem} \label{lem:max-weight-group} (Max Weight Matching)  
Every slot $t$, the  LCQ algorithm chooses a transmission rate vector $\bv{\mu}(t) \in \feasible(t)$
that maximizes the following expression over all alternative feasible transmission rate
vectors: 
\[ \sum_{g \in \script{G}_K} \sum_{k=1}^K Q_k^{(g)}(t) \tilde{\mu}_k^{(g)}(t) \]
\end{lem}
\begin{proof} 
See Appendix B.
\end{proof} 

It follows that we can replace the variables $\tilde{\mu}_k^{(g)}(t)$ in the final term of
the drift expression in Lemma \ref{lem:drift-group}, which correspond
to the LCQ policy, with variables $\tilde{\mu}_k^{(g)*}(t)$ that correspond to any
other feasible rate vector $\bv{\mu}^*(t) \in \feasible(t)$, while creating an inequality
relationship: 
\begin{eqnarray}
\hspace{-.2in}\Delta(\bv{Q}(t)) &\leq& \expect{C(t)\left|\right.\bv{Q}(t)} \nonumber \\
&& \hspace{-.3in} - \sum_{g \in \script{G}_K}\sum_{k=1}^K Q_k^{(g)}(t)\expect{\tilde{\mu}_k^{(g)*}(t) - \lambda_k^{(g)}\left|\right.\bv{Q}(t)} \label{eq:group-drift-final}
\end{eqnarray}

The drift inequality (\ref{eq:group-drift-final}) is quite subtle:  It is defined
in terms of any other \emph{single} feasible rate vector $\bv{\mu}^*(t)$ (where this
vector does not depend on the partition $g$). 
Note that the variables $\tilde{\mu}_k^{(g)*}(t)$ are defined for different partitions $g\in\script{G}_K$, 
but for each particular $g$ these variables 
are still derived from the \emph{same}  vector 
$\bv{\mu}^*(t)$.  They are derived from $\bv{\mu}^*(t)$ 
by summing the components of this rate vector that have 
non-empty queues over the dimensions that correspond to 
the groups within the particular partition $g$.

\subsection{Optimizing the Drift Bound} 

Here we manipulate the sum in the right-hand side of (\ref{eq:group-drift-final})
to yield a useful drift bound. 

\begin{lem} \label{lem:group-lambda}
For any vector $\bv{\lambda} = (\lambda_1, \ldots, \lambda_{\hat{N}})$,  
if there is a value $\beta$
such that $0 < \beta < 1$ such that for all $i \in \{1, \ldots, N\}$ we have: 
\begin{equation} \label{eq:f-balance1}
\lambda_i \leq \frac{\lambda_{tot}}{\hat{N}} + \frac{\beta(1-\rho)}{K}  
\end{equation}
then: 
\begin{eqnarray*}
 \sum_{g\in\script{G}_K}\sum_{k=1}^K Q_k^{(g)}(t)\lambda_k^{(g)} &\leq& Q_{tot}(t)\left|\script{G}_K\right|\frac{[\lambda_{tot} + z\beta(1-\rho)]}{K} 
  \end{eqnarray*}
 where $\left|\script{G}_K\right|$ is the cardinality of $\script{G}_K$, $Q_{tot}(t)$
is the total sum backlog (defined in (\ref{eq:q-tot})), and $z$ is defined: 
\begin{equation} \label{eq:z} 
z \defequiv (1-1/K)/(1 - 1/\hat{N})
\end{equation}
\end{lem}
\begin{proof}
The proof of Lemma \ref{lem:group-lambda} follows from simple counting arguments, 
and is given in Appendix C. 
\end{proof}

Note that $\lambda_{tot}/\hat{N} \leq \lambda_{tot}/N$ with approximate equality 
when $N$ is large (so that $\hat{N}/N \approx 1$). 
The constraints (\ref{eq:f-balance1}) imply that $\bv{\lambda}$ is $f$-balanced with $f = \beta(1-\rho)/K$.

\begin{lem} \label{lem:group-mu} There exists a single randomized
strategy that observes queue backlogs and channel states for slot $t$ and chooses  
$\bv{\mu}^*(t) \in \feasible(t)$ such that: 
\begin{eqnarray*}
\sum_{g \in \script{G}_K} \sum_{k=1}^K Q_k^{(g)}(t)\expect{\tilde{\mu}_k^{(g)*}(t)\left|\right.\bv{Q}(t)}
\geq Q_{tot}(t)\left|\script{G}_K\right|\frac{r_K}{K}
\end{eqnarray*}
where $r_K = 1 - (1-p_{min})^K$. 
\end{lem} 
\begin{proof}
See Appendix C.
\end{proof}

 Using Lemmas \ref{lem:group-lambda} and \ref{lem:group-mu}
in the drift inequality (\ref{eq:group-drift-final}) yields: 

\begin{lem} \label{lem:penultimate} 
If $\bv{\lambda} \in \rho \Lambda$ (for $0 < \rho < 1$) and if (\ref{eq:f-balance1}) 
is satisfied for all $i \in \{1, \ldots, N\}$, then: 
\begin{eqnarray}
\Delta(\bv{Q}(t)) &\leq& \expect{C(t)\left|\right.\bv{Q}(t)} \nonumber \\
 && \hspace{-.1in} - Q_{tot}(t)\left|\script{G}_K\right| \frac{[r_K - \lambda_{tot} - z\beta(1-\rho)]}{K} 
 \label{eq:onoff2-drift-last}
\end{eqnarray}
\end{lem}

Lemma \ref{lem:penultimate} leads immediately to the delay theorem stated in the next
subsection. 

\subsection{An Improved Delay Bound for $f$-Balanced Traffic} 

\begin{thm} \label{thm:onoff2} (Delay Bound for ON/OFF Channels with $f$-Balanced
Traffic)  
Suppose $\bv{\lambda} \in \rho\Lambda$
for $0 < \rho < 1$.  Let $K$ be the smallest integer that satisfies $r_K \geq (1+\rho)/2$, that is: 
\begin{equation} \label{eq:K-onoff2} 
[1 - (1-p_{min})^K] \geq (1+\rho)/2 
\end{equation} 
Suppose that $K \leq N$, and 
the $f$-balanced traffic constraints (\ref{eq:f-balance1}) are satisfied for some
value $\beta$ such that $0 \leq \beta < 1/(2z)$, where $z \defequiv (1 - 1/K)/(1 - 1/\hat{N})$ (note that $z \leq 1$).  
If the max-weight (LCQ) policy is used on this ON/OFF channel model, then average queue occupancy
satisfies: 
\begin{eqnarray}
\overline{Q}_{tot} \leq \frac{KD}{(1-\rho)(\frac{1}{2} - z\beta)} \: \: \: , \: \: \: \overline{W} \leq  \frac{K(D/\lambda_{tot})}{(1-\rho)(\frac{1}{2} - z\beta)} \label{eq:onoff2} 
\end{eqnarray}
where $D$ is defined: 
\begin{eqnarray*}
D \defequiv \frac{1}{2} \left[  \lambda_{tot} + \expect{A_{tot}(t)^2}   \right] 
\end{eqnarray*}
Further,  
in the special case when $N$ is a multiple of $K$, and when traffic is uniform and 
Poisson 
with  $\lambda_i = \lambda_{tot}/N$ for all $i$, we have $\beta = 0$ and:\footnote{These  bounds for
symmetric Poisson traffic are obtained from the last line of the proof of Theorem \ref{thm:onoff2}, 
which gives a slightly
smaller bound than that achieved by plugging $\beta = 0$, $\expect{A_{tot}(t)^2} = \lambda_{tot} + \lambda_{tot}^2$
into (\ref{eq:onoff2}).}
\[ \overline{Q}_{tot} \leq   \frac{[2K\lambda_{tot}   -  \lambda_{tot}^2]}{1-\rho} \:  \: \: , \: \:  \: 
\overline{W} \leq \frac{2K -  \lambda_{tot}}{1-\rho} \]
\end{thm}

Note that the constraint (\ref{eq:K-onoff2}) is satisfied by: 
\[ K = \left\lceil \frac{\log(\frac{2}{1-\rho})}{\log(1/(1-p_{min}))}  \right\rceil \]
Therefore, $K$ is independent of $N$, and is proportional to $\log(\frac{1}{1-\rho})$. Assuming that 
traffic streams are independent, so that 
$\expect{A_{tot}(t)^2} = O(1)$,  implies that $D = O(1)$.  Thus  the delay bound
gives $\overline{W} \leq c\frac{\log(1/(1-\rho))}{1-\rho}$ (where $c$ is a constant independent of $\rho$ and 
$N$),  being independent of the 
network size $N$ and having an asymptotic in $\rho$ that is better than that of 
Theorem \ref{thm:onoff1}.  

\begin{proof} (Theorem \ref{thm:onoff2}) 
Because $\bv{\lambda} \in \rho\Lambda$, we have $\lambda_{tot} \leq \rho$ (as the maximum
sum rate is at most $r_N \leq 1$).  The assumption on $r_K$ in (\ref{eq:K-onoff2}) thus implies: 
\[ [r_K - \lambda_{tot} - z\beta(1-\rho)] \geq (1-\rho)(\frac{1}{2} - z\beta) \]
The above value is strictly positive because $z\beta < 1/2$.  
Using the drift inequality (\ref{eq:onoff2-drift-last}) directly in the  
Lyapunov Drift Lemma  (Lemma \ref{lem:lyap-drift}) yields:
\begin{eqnarray*}
\overline{Q}_{tot} \leq \frac{K\overline{C}}{\left|\script{G}_K\right|(1-\rho)(\frac{1}{2} - z\beta)}
\end{eqnarray*}
Using the definition of $C(t)$ in (\ref{eq:c-t-onoff2}) and the fact that the system
is stable (so the long term departure rate is equal to $\lambda_{tot}$) yields: 
\begin{eqnarray*}
\overline{C} &=& \frac{\left|\script{G}_K\right|\lambda_{tot}}{2} + 
\frac{1}{2} \sum_{g \in \script{G}_K}\sum_{k=1}^K \left[\expect{A_k^{(g)}(t)^2} - 2(\lambda_k^{(g)})^2\right] \\
&\leq& \left|\script{G}_K\right| \left[ \frac{\lambda_{tot}}{2} + \frac{\expect{A_{tot}(t)^2}}{2}     \right]  = \left|\script{G}_K\right|D
\end{eqnarray*}
The above bound on $\overline{C}$ proves the first part of the theorem.  The second part, for
uniform Poisson traffic, follows by the above equality for $\overline{C}$ (without the bound), 
using $\expect{A_k^{(g)}(t)} = \frac{\lambda_{tot}}{K}$ and $\expect{A_k^{(g)}(t)^2} = 
\frac{\lambda_{tot}^2}{K^2} + \frac{\lambda_{tot}}{K}$ for all $g$, $k$. 
\end{proof}

\section{Multi-Rate Transmission Models} \label{section:multi-rate} 

Now suppose that for each channel $i \in \{1, \ldots, N\}$, the 
states $S_i(t)$ are non-negative integers bounded by a finite integer 
$\mu_{i, max}$, where  $\mu_{i, max}$ represents the maximum transmission
rate over channel $i$.\footnote{For consistency, 
we continue to work in integer units of packets. The analysis
does not significantly change if $S_i(t)$ values are viewed as non-negative 
real numbers with units of bits/slot.}  That is, we have: 
\[ S_i(t) \in \{0, 1, \ldots,  \mu_{i, max}\} \: \: \mbox{ for all $t$ and all $i \in \{1, \ldots, N\}$} \]
We assume that $\mu_{i,max} > 0$ for all $i$.  
The queueing dynamics are governed by (\ref{eq:dynamics}).   The capacity region $\Lambda$ is known
to be equal to the set of all rate vectors that can be achieved via a stationary, randomized, queue-independent
algorithm that chooses $\bv{\mu}^*(t)$  as a potentially random function of only the current $\bv{S}(t)$ 
vector \cite{now}.  

The max-weight algorithm in this case is the algorithm that observes queue backlogs and channel states
every slot and selects the link $i \in \{1, \ldots, N\}$ with the largest value of $Q_i(t)S_i(t)$ (breaking
ties arbitrarily). 
Suppose the arrival rate vector satisfies $\bv{\lambda} \in \rho\Lambda$ for some
loading value $\rho$ such that $0 < \rho < 1$.  
The analysis in 
\cite{neely-downlink-ton} \cite{now}  uses a standard Lyapunov function, given by the sum of the squares
of queue backlog, to show the max-weight algorithm for a general downlink has average delay
upper bounded by $cN/(1-\rho)$, where $c$ is a constant that is independent of $N$ and $\rho$.
We first present a modified version of that prior bound, which has the same structure but
uses our particular $\mu_{i,max}$ notation and improves the $c$ coefficient:   
\begin{lem} \label{lem:multi1} Suppose $\bv{A}(t)$ is i.i.d. over slots with $\expect{\bv{A}(t)} = \bv{\lambda}$, 
and that the channel state vector  $\bv{S}(t) = (S_1(t), \ldots, S_N(t))$ is  
also i.i.d. over slots.  Suppose that $\bv{\lambda} \in \rho \Lambda$ for some value $\rho$ that satisfies
$0 < \rho < 1$.  Then the system is stable under the max-weight algorithm and has 
an average delay bound given by:
\begin{eqnarray}
\overline{W} &\leq&   \frac{N\left[\frac{1}{2\lambda_{tot}}\sum_{i=1}^N \expect{A_i^2} - \frac{3}{2\lambda_{tot}}\sum_{i=1}^N \lambda_i^2\right] }{(1-\rho)\mu_{sym}}  \nonumber \\
&& + \frac{N\min\left[   \sum_{i=1}^N \frac{\lambda_i \mu_{i,max}}{\lambda_{tot}} , \frac{\hat{S}^2}{\lambda_{tot}}  \right]}{(1-\rho)\mu_{sym}}  \label{eq:original-multi} 
\end{eqnarray} 
where $\hat{S}^2$ is defined: 
\[ \hat{S}^2 \defequiv \expect{\max_{i\in\{1, \ldots, N\}} S_i(t)^2} \]
and where $\mu_{sym}$ is defined as the largest value such that $(\mu_{sym}/N, \mu_{sym}/N, \ldots, \mu_{sym}/N) \in \Lambda$.  

Further, in the case when all processes $S_i(t)$ are independent and satisfy $Pr[S_i(t) = \mu_{i,max}] >0$, we can bound
$\mu_{sym}$ as follows: 
\[ \mu_{sym} \geq \hat{\mu}(1-(1-p_{min})^N) \]
where  $\hat{\mu}$ and $p_{min}$ are defined: 
\begin{eqnarray*}
\hat{\mu} &\defequiv& \min_{i \in \{1, \ldots, N\}} \mu_{i,max} \\
p_{min} &\defequiv& \min_{i\in\{1, \ldots, N\}} Pr[S_i(t) \geq \hat{\mu}]
\end{eqnarray*}
\end{lem} 
\begin{proof} 
The proof uses the Lyapunov function 
$L(\bv{Q}(t)) = \frac{1}{2}\sum_{i=1}^N Q_i(t)^2$, as in \cite{neely-downlink-ton} \cite{now}, but provides
a simple modification of the argument to yield a tighter bound (given in Appendix E for completeness).
\end{proof}

We note that the above delay bound holds also in the case when $\mu_{i,max} = \infty$ for some
values $i$, but when the second moment of transmission rates is finite (so that $\hat{S}^2$ is finite). 
The above delay bound has the structure $cN/(1-\rho)$, and holds even if arrival and 
channel vectors $\bv{A}(t)$ and $\bv{S}(t)$ have entries that are correlated over the different links
$i \in \{1, \ldots, N\}$.   A similar argument can be used to show stability with the same structural
delay bound $\hat{c}N/(1-\rho)$ 
for the modified max-weight policy that chooses the link $i \in \{1, \ldots, N\}$ with the 
largest $Q_i(t)\min[Q_i(t), S_i(t)]$ value. This modified policy can sometimes provide smaller
empirical average delay than the original max-weight policy, although its resulting analytical 
delay bound has a slightly worse coefficient $\hat{c} \geq c$ (this modified policy is equivalent to the original
max-weight policy in the case of ON/OFF channels with $\mu_{i,max} = 1$ for all $i$).  Similar to the ON/OFF
case, one might suspect that for this multi-rate system,  average delay 
that is independent of $N$ can be achieved when arrival and channel processes are independent
over each channel.  However, the next subsection presents 
an important example that shows this is not the case.\footnote{We note that our original
pre-print of this paper in \cite{neely-maxweight-delay-arxiv}
 incorrectly claimed that multi-rate systems also have delay that is 
 independent of $N$.   The mistake in \cite{neely-maxweight-delay-arxiv} arose
 when  plugging the equation from Lemma 10 of that paper into equation (33) of that paper.
 Plugging one equation into the other implicitly  
 assumed that  the sum  queue backlog in queues with at least $\mu_{max}$ packets
 is the same as the total queue backlog in the system.  This is true when $\mu_{max} = 1$, but is 
 not true in general as it neglects
 the ``residual'' packets in queues with fewer than $\mu_{max}$
 packets.}

\subsection{An example showing necessity of $O(N)$ delay} \label{section:example} 

Here we present an example showing that the average number of queues that have 
at least one packet but fewer than $\mu_{i,max}$ packets must be linear in $N$, which necessarily
makes the average delay of \emph{any} scheduling policy grow at least linearly with $N$. 
Consider a system with $N$ queues with symmetric channels and traffic.  Assume that $N \geq 3$ 
and suppose that all arrival processes $A_i(t)$ are independent and Bernoulli with $Pr[A_i(t) = 1] = 3/N$
for all $i \in \{1, \ldots, N\}$ (so that $\lambda_i = 3/N$ for all $i$, and $\lambda_{tot} = 3$ packets/slot). 
Now suppose that all channels have $\mu_{i,max} = 5$, and  channel 
state processes are i.i.d. with $Pr[S_i(t) = 5] = 1/2$, $Pr[S_i(t) = 0] = 1/2$ for all $i \in \{1, \ldots, N\}$. 
The largest symmetric rate in the capacity region of this system is $\mu_{sym}/N = 5(1 -(1/2)^N)/N$, 
and hence the arrival rate vector is inside the capacity region and has $\rho$ given by: 
\[ \rho = \frac{3}{5(1-(1/2)^N)} \]
Note that $\rho < 1$ for $N \geq 3$, and $\rho$ is approximately $3/5$ for large $N$. 
Here we show that under any scheduling policy, in steady state the average number of non-empty
queues in this system must be linear in $N$.   Specifically, consider any scheduling policy, and let $Z(t)$ represent
the number of non-empty queues on slot $t$.  For simplicity, we assume that $Z(t)$ has a well
defined steady state under the scheduling policy.   
The intuition behind our proof is that $Z(t+1)$ is formed from $Z(t)$ by adding the number of new
non-empty queues created and subtracting any non-empty queue that becomes empty.  
The number of non-empty queues subracted 
can be at most  1 (as we can serve at most one channel per slot), while  the average number
of new non-empty queues added  is 
\emph{more than one} whenever $Z(t) < N/2$. 

\begin{lem} Consider any scheduling policy for which $Z(t)$ has a well defined steady state
distribution.  Then for the system above (with $\lambda_i = 3/N$ and $\mu_{i,max} = 5$ for
all $i\in\{1, \ldots, N\}$) we have that in steady state: 
\[Pr[Z(t) \geq N/2] \geq 1/3 \]
and hence $\expect{Z(t)} \geq N/6$. That is, the average number of non-empty queues is at least $N/6$, 
and hence the average number of packets in the system is at least $N/6$. 
\end{lem} 
\begin{proof}
Define $\Delta(t) \defequiv Z(t+1) - Z(t)$ as the change in $Z(t)$ from one slot to the next.  Let $t$ be a time
at which the  system is in steady 
state.  We thus have $\expect{\Delta(t)} = 0$.   On the other hand, we have the following: 
\begin{eqnarray}
\expect{\Delta(t) \left|\right.Z(t) \geq  N/2} &\geq&  -1 \label{eq:counter1} \\
\expect{\Delta(t) \left|\right. Z(t) < N/2} &\geq& \lambda_i N/2 - 1 = 1/2 \label{eq:counter2} 
\end{eqnarray}
where (\ref{eq:counter1}) follows because the drift cannot be less than $-1$ on any slot (as at most
one non-empty queue can become an empty queue), and (\ref{eq:counter2}) holds because, given
that $Z(t) < N/2$, the
average number of \emph{new} non-empty queues that are created on slot $t$ is equal to the average
number of new arrivals to the empty queues, which is at least $\lambda_i (N/2)$.  It follows that: 
\begin{eqnarray*}
0 &=& \expect{\Delta(t)}  \\
&=& \expect{\Delta(t)\left|\right.Z(t) \geq N/2}Pr[Z(t) \geq N/2] \\
 && + \expect{\Delta(t)\left|\right.Z(t) < N/2}(1-Pr[Z(t) \geq N/2]) \\
 &\geq& (-1)Pr[Z(t) \geq N/2] \\
 && + (1/2)(1-Pr[Z(t) \geq N/2]) 
\end{eqnarray*}
Therefore $Pr[Z(t) \geq N/2] \geq 1/3$, completing the proof.
\end{proof}

\section{Simulations} \label{section:sim}

Here we present simulations for the ON/OFF system with independent channel and arrival
processes. 
We assume that $Pr[S_i(t) = ON] = 1/2$ for all $i \in \{1, \ldots, N\}$.  We 
first consider symmetric Bernoulli arrivals, so that $\lambda_i = \lambda$ for all
$i$, where  $\lambda$ is chosen so that $\bv{\lambda} \in \rho \Lambda$ with 
 $\rho = 0.8$.  We simulated the system over 
$10^6$ slots for values of $N$ between $3$ and $300$.  The resulting simulated
queue averages are shown in Fig. \ref{fig:1}, together with the two $O(1)$ bounds 
and the previous $O(N)$ bound. 
Note that the previous $O(N)$ bound is a considerable overestimate of queue backlog.
Our new bounds do not grow with $N$, and our second $O(1)$ bound (derived for
$f$-balanced traffic rates)  is indeed tighter
than the first $O(1)$ bound (for $N\geq 9$),  although it applies only to $f$-balanced traffic
while the first bound applies to any traffic rates in $\rho \Lambda$. However,
there is still a significant gap (roughly a factor of 10 in this example) 
between our tightest bound and the simulated value.
We next consider heterogeneous traffic rates implemented on the same ON/OFF system.
We assume that $N$ is odd, and choose rates $\lambda_i$ given as follows: 

\[ \lambda_i = \left\{ \begin{array}{ll}
                          \lambda &\mbox{ for $i \in \{1, \ldots, (N-1)/2\}$} \\
                             2\lambda  & \mbox{ for $i \in \{(N-1)/2 + 1, \ldots, N-1\}$} \\
                             4\lambda & \mbox{ for $i = N$} 
                            \end{array}
                                 \right.  \]
where $\lambda$ is chosen so that $\bv{\lambda} \in \rho \Lambda$ for $\rho = 0.8$. The 
results are shown in Fig. \ref{fig:2}. Note that we plot only the first $O(1)$ bound (for heterogeneous
traffic) in this case, although the $f$-balanced traffic assumption also applies in this case when $N$ is 
sufficiently large. 

Simulations of the multi-rate system example in Section \ref{section:example} were
also conducted, and it was verified that average backlog indeed grows linearly with $N$ due to the ``residual''
packets in  queues $i$ that have  fewer than $\mu_{i,max}$ packets (simulation 
plots omitted for brevity).  
However, it was observed in the simulations that the total backlog due to queues 
with \emph{at least}  $\mu_{i,max}$ 
packets is $O(1)$.  This suggests that, although the total average backlog in 
multi-rate systems may have a fundamental  $O(N)$ term due to residual packets, 
the average backlog may be $O(1)$ after a term of at most 
$\sum_{i=1}^N (\mu_{i,max} -1)$ is subtracted out. 

\begin{figure}[htbp]
   \centering
   \includegraphics[height=2.5in, width=3in]{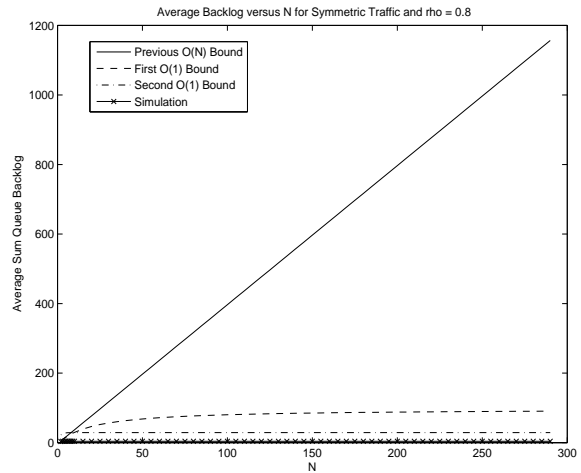} 
   \caption{Simulation and Bounds for the ON/OFF system with symmetric traffic and $\rho = 0.8$.}
   \label{fig:1}
\end{figure}

\begin{figure}[htbp]
   \centering
   \includegraphics[height=2.5in, width=3in]{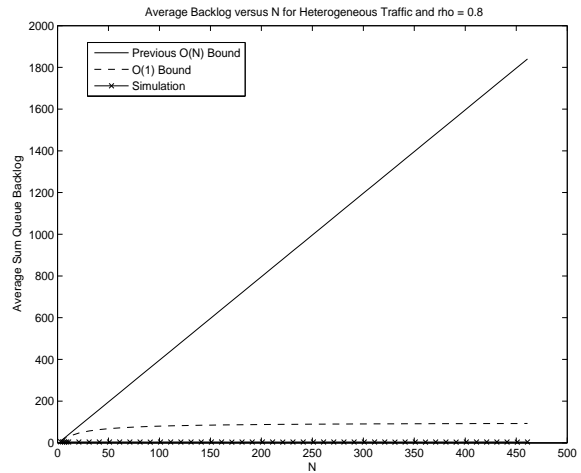} 
   \caption{Simulation and Bounds for the ON/OFF system with heterogeneous traffic and $\rho = 0.8$.}
   \label{fig:2}
\end{figure}

\section{Conclusions} 

We have presented an improved delay analysis for the max-weight scheduling 
algorithm.  For ON/OFF channels, max-weight is equivalent to Longest Connected
Queue (LCQ), and yields average delay that is order-optimal, being independent of the 
network size $N$.  If an $f$-balanced traffic assumption holds, average delay was shown
to maintain independence of $N$ while allowing an improved asymptotic in $\rho$. 
For multi-rate channels, a delay bound of $O(N)$ applies.  Conversely, it 
is shown for a simple
multi-rate example that, unlike ON/OFF channels, average backlog must be \emph{at least} linear
in $N$ due to ``residual'' packets. 
   Our delay analysis makes use
of the technique of queue grouping.  The particular
Lyapunov functions introduced for this delay analysis are powerful 
and may be useful in
other contexts. 

\section*{Appendix A --- Proof of Lemma \ref{lem:drift1}}

Here we prove Lemma \ref{lem:drift1}.  
 Define $\Delta_1(\bv{Q}(t))$ and $\Delta_2(\bv{Q}(t))$ as the conditional drift
 for the sum of squares term and the square of queue backlog term, respectively, 
 so that $\Delta(\bv{Q}(t)) = \Delta_1(\bv{Q}(t)) + \theta\Delta_2(\bv{Q}(t))$. 
 Squaring (\ref{eq:dynamics-tilde}) and using the fact that $\tilde{\mu}_i(t)^2 = \tilde{\mu}_i(t)$
and $Q_i(t)\tilde{\mu}_i(t) = Q_i(t)\mu_i(t)$ (because $\tilde{\mu}_i(t) \in \{0, 1\}$, and
$\tilde{\mu}_i(t) = \mu_i(t)$ if $Q_i(t)>0$): yields: 
 \begin{eqnarray*} 
 \frac{1}{2}Q_i(t+1)^2 &=& \frac{1}{2} Q_i(t)^2  
  + \frac{(A_i(t) - \tilde{\mu}_i(t))^2}{2} \\
  && - Q_i(t)(\mu_i(t) - A_i(t))\\
  &=& \frac{1}{2} \left[Q_i(t)^2  + A_i(t)^2 + \tilde{\mu}_i(t)\right] - A_i(t)\tilde{\mu}_i(t) \\
  && - Q_i(t)(\mu_i(t) - A_i(t)) 
 \end{eqnarray*} 
 Therefore: 
 \begin{eqnarray*}
 \Delta_1(\bv{Q}(t)) &=& \expect{B_1(t)\left|\right.\bv{Q}(t)} \\
 &&  - \sum_{i=1}^NQ_i(t)\expect{\mu_i(t)- \lambda_i\left|\right.\bv{Q}(t)} 
 \end{eqnarray*}
 where
 \[ B_1(t) \defequiv \sum_{i=1}^N\left[\frac{1}{2} \left[ A_i(t)^2 + \tilde{\mu}_i(t)\right] - A_i(t)\tilde{\mu}_i(t)\right] \] 
 
 Similarly: 
 \begin{eqnarray*}
 \frac{1}{2}Q_{tot}(t+1)^2 &=& \frac{1}{2}\left[Q_{tot}(t)^2 + \tilde{\mu}_{tot}(t)^2 + A_{tot}(t)^2\right] \\
 && - \tilde{\mu}_{tot}(t)A_{tot}(t) \\
 && - Q_{tot}(t)(\tilde{\mu}_{tot}(t) - A_{tot}(t))
 \end{eqnarray*}
 Therefore: 
 \begin{eqnarray*}
 \Delta_2(\bv{Q}(t)) &=& \expect{B_2(t)\left|\right.\bv{Q}(t)} \\
 && - Q_{tot}(t)\expect{\tilde{\mu}_{tot}(t) - \lambda_{tot}\left|\right.\bv{Q}(t)} 
 \end{eqnarray*}
 where $\tilde{\mu}_{tot}(t)^2 = \tilde{\mu}_{tot}(t)$ (because it is either $0$ or $1$), and 
 \begin{eqnarray*}
 B_2(t) \defequiv \frac{1}{2}\left[ \tilde{\mu}_{tot}(t) + A_{tot}(t)^2\right] 
  - \tilde{\mu}_{tot}(t)A_{tot}(t) \\
 \end{eqnarray*}
 Summing $\Delta_1(\bv{Q}(t))$ and $\theta\Delta_2(\bv{Q}(t))$ and noting from
 (\ref{eq:bt}) 
 that $B(t) = B_1(t) + \theta B_2(t)$ yields the result of Lemma \ref{lem:drift1}. 

\section*{Appendix B --- Proof of Lemma \ref{lem:max-weight-group}}
Define the integer $m = \hat{N}/K$. 
Here we prove Lemma \ref{lem:max-weight-group}. 
Given a particular queue backlog vector $\bv{Q}(t)$,  
the LCQ algorithm maximizes the expression $\sum_{i=1}^{\hat{N}} Q_i(t) \mu_i(t)$ over
all  $\bv{\mu}(t) \in \feasible(t)$. We now show that this \emph{also} maximizes 
the expression given in Lemma \ref{lem:max-weight-group}.  
To this end, we have: 
\begin{eqnarray*}
\sum_{g \in \script{G}_K}\sum_{k=1}^K Q_k^{(g)}(t) \tilde{\mu}_k^{(g)}(t) &=& 
\sum_{g \in \script{G}_K}\sum_{k=1}^K Q_k^{(g)}(t)\sum_{i \in \script{L}_k^{(g)}}\tilde{\mu}_i(t) \\
&&  \hspace{-.3in} = \sum_{i=1}^{\hat{N}} \tilde{\mu}_i(t) Q_i(t)\left|\script{G}_K\right|\\
&& \hspace{-.3in} + \sum_{i=1}^{\hat{N}} \tilde{\mu}_i(t) \sum_{j\neq i} Q_j(t)\frac{\left|\script{G}_K\right|(m-1)}{\hat{N}-1}  
\end{eqnarray*}
where the final equality holds because link $i$ is in every group that multiplies
the $\tilde{\mu}_i(t)$ term, and all other links multiply this term the same number of times
(by group symmetry).  The above also uses the fact that (by symmetry) 
 the number of group partitions for which 
a particular link $j$ is in the same group as link
 $i$ is equal to the total number of partitions multiplied by the probability that a randomly chosen partition includes $i$ and $j$ in the same group. 
Define the above expression as $f(t)$ for simplicity. Therefore: 
\begin{eqnarray}
f(t) &=&   \sum_{i=1}^{\hat{N}} \tilde{\mu}_i(t) Q_i(t)\left|\script{G}_K\right|(1 - \frac{m-1}{\hat{N}-1})  \nonumber \\
&& + \sum_{i=1}^{\hat{N}}\tilde{\mu}_i(t) \left(\sum_{j=1}^{\hat{N}} Q_j(t)\right)\left|\script{G}_K\right|\frac{m-1}{\hat{N}-1}  \label{eq:sum-swapping} 
\end{eqnarray}
The $\tilde{\mu}_i(t)$ values in the expression for $f(t)$ are the 
only ones affected by the control action on slot $t$. 
The final term on the right hand side is given by $\sum_i \tilde{\mu}_i(t)$ (the total departures
on slot $t$) multiplied by a non-negative constant.  
This final term is maximized by any work conserving policy that always transmits a packet when
there is a non-empty connected queue.  The first term on the right hand side is a non-negative
constant multiplied by the term $\sum_i \tilde{\mu}_i(t) Q_i(t)$.  But note that
$\tilde{\mu}_i(t) Q_i(t) = \mu_i(t) Q_i(t)$, and thus the LCQ policy maximizes this first
term.  As LCQ is work conserving, it also maximizes the second term, and thus
maximizes $f(t)$, proving Lemma \ref{lem:max-weight-group}.

\section*{Appendix C --- Proof of Lemmas \ref{lem:group-lambda} and \ref{lem:group-mu}}


\begin{proof} (Lemma \ref{lem:group-lambda}) Define the integer
$m = \hat{N}/K$. Using a counting argument similar
to that of Appendix B (compare with (\ref{eq:sum-swapping})), 
we have that 
for any vector $\bv{\lambda} = (\lambda_1, \ldots, \lambda_{\hat{N}})$: 
\begin{eqnarray}
\sum_{g\in\script{G}_K}\sum_{k=1}^K Q_k^{(g)}(t)\lambda_k^{(g)} &=& \sum_{i=1}^{\hat{N}} Q_i(t)\lambda_i\left|\script{G}_K\right|(1-\frac{m-1}{{\hat{N}}-1}) \nonumber \\
&& + Q_{tot}(t)\lambda_{tot} \left|\script{G}_K\right|\frac{m-1}{{\hat{N}}-1} \label{eq:lambda-derivation}
\end{eqnarray}
Using the bound on $\lambda_i$ given in (\ref{eq:f-balance1}) yields: 
\begin{eqnarray*}
 \sum_{g\in\script{G}_K}\sum_{k=1}^K Q_k^{(g)}(t)\lambda_k^{(g)} \leq \hspace{+1.5in}&\\
  Q_{tot}(t)\lambda_{tot}\left|\script{G}_K\right|\left[\frac{1}{{\hat{N}}} - \frac{m-1}{{\hat{N}}({\hat{N}}-1)} + \frac{m-1}{{\hat{N}}-1}\right] & \\
 +  Q_{tot}(t)\left|\script{G}_K\right|(1-\frac{m-1}{{\hat{N}}-1})\beta(1-\rho)/K & 
 \end{eqnarray*}
 The result of Lemma \ref{lem:group-lambda} follows by  using the identities:
 \begin{equation} \label{eq:identity} 
  \left[\frac{1}{{\hat{N}}} - \frac{m-1}{{\hat{N}}({\hat{N}}-1)} + \frac{m-1}{{\hat{N}}-1}\right] = \frac{1}{K} 
  \end{equation}
  \[ \left[ 1 - \frac{m-1}{\hat{N}-1}\right] = z \]
\end{proof}

\begin{proof} (Lemma \ref{lem:group-mu}) 
Let $L(t)$ represent the number of non-empty queues on slot $t$.  
If $L(t) = 0$, then $Q_{tot}(t)= 0$ and the result is trivial.  
Now suppose that $L(t) = l$, where $l \in \{1, 2, \ldots, N\}$.  Define $(l_1, \ldots, l_{\hat{N}})$ to be 
a $0/1$ vector with $l_i = 1$ if and only if $Q_i(t)>0$.   Define $l_k^{(g)}$ to be the number
of non-empty queues in the $k$th group of partition $g$.  
Consider the following randomized policy for $\bv{\mu}^*(t) \in \feasible(t)$: 
First observe all channel states $S_i(t)$ for non-empty queues $i$, and define 
\emph{new} channel states $\hat{S}_i(t)$ as follows: 
If $S_i(t) = 0$ (OFF), assign $\hat{S}_i(t) = 0$.  If 
$S_i(t) = 1$ (ON), independently assign $\hat{S}_i(t) = 1$ with probability $p_{min}/p_i$ (this is 
a valid probability because $p_{min} \leq p_i$).  It follows that the new channel state
vector $\hat{\bv{S}}(t)$ has independent and symmetric ON probabilities $p_{min}$.  
Now independently, randomly, and uniformly choose a queue to serve over all 
non-empty queues $i$ with $\hat{S}_i(t)=1$.   It follows that for all non-empty queues $i$ we have: 
\[ \expect{\tilde{\mu}_i^*(t)\left|\right.\bv{Q}(t)}  = \frac{1 - (1-p_{min})^l}{l} = \frac{r_l}{l}  \]
Further, for any $g \in \script{G}_K$ and any $k \in \{1, \ldots, K\}$ we have: 
\[ \expect{\tilde{\mu}_k^{(g)*}(t)\left|\right.\bv{Q}(t)} = \sum_{i \in \script{L}_k^{(g)}} \expect{\tilde{\mu}_i^*(t)\left|\right.\bv{Q}(t)} = l_k^{(g)}\frac{r_l}{l} \]
Using this equality gives: 
\begin{eqnarray}
\sum_{g\in\script{G}_K}\sum_{k=1}^K Q_k^{(g)}(t)\expect{\tilde{\mu}_k^{(g)*}(t)\left|\right.\bv{Q}(t)} = \nonumber  \\
\frac{r_l}{l}\sum_{g \in \script{G}_K}\sum_{k=1}^K Q_k^{(g)}(t)l_k^{(g)} \label{eq:appc1}
\end{eqnarray}
Now note that the $l_k^{(g)}$ values are structurally 
similar to the $\lambda_k^{(g)}$ values, and hence (similar to (\ref{eq:lambda-derivation})) we have
(using $Q_i(t)l_i = Q_i(t)$ and $l_{tot} =  l$): 
\begin{eqnarray*}
 \sum_{g \in \script{G}_K}\sum_{k=1}^K Q_k^{(g)}(t)l_k^{(g)} &=&  \sum_{i=1}^{\hat{N}} Q_i(t) \left|\script{G}_K\right|(1-\frac{m-1}{\hat{N}-1}) \\
 && + Q_{tot}(t) l \left|\script{G}_K\right|\frac{m-1}{\hat{N}-1} 
 \end{eqnarray*}

Using this in (\ref{eq:appc1}) yields: 
\begin{eqnarray}
\sum_{g\in\script{G}_K}\sum_{k=1}^K Q_k^{(g)}(t)\expect{\tilde{\mu}_k^{(g)*}(t)\left|\right.\bv{Q}(t)} \nonumber \\
= 
\frac{r_l}{l} Q_{tot}(t)\left|\script{G}_K\right|\left[ 1 - \frac{m-1}{\hat{N}-1} + \frac{l(m-1)}{\hat{N}-1}  \right] \label{eq:appc-reuse} \\
\geq r_lQ_{tot}(t)\left|\script{G}_K\right|\left[ \frac{1}{\hat{N}} -   \frac{(m-1)}{\hat{N}(\hat{N}-1)} + \frac{m-1}{\hat{N}-1}    \right] \nonumber  \\
= r_l Q_{tot}(t) \left|\script{G}_K\right|/K \label{eq:appc-reuse2}
\end{eqnarray}
where the last equality holds by (\ref{eq:identity}). 
The above holds for $L(t) = l \in \{1, \ldots, N\}$.   Suppose now that $l \geq K$.  In this
case we have  $r_l \geq r_K$, proving the result of  Lemma \ref{lem:group-mu} for $l \geq K$. 

Consider now the final case where  $l \in \{1, \ldots, K-1\}$. Then from (\ref{eq:appc-reuse}) we have: 
\begin{eqnarray}
\sum_{g\in\script{G}_K}\sum_{k=1}^K Q_k^{(g)}(t)\expect{\tilde{\mu}_k^{(g)*}(t)\left|\right.\bv{Q}(t)} \nonumber \\
\geq  \frac{r_l}{l} Q_{tot}(t) \left|\script{G}_K\right|
\end{eqnarray}
Using the fact that 
$\frac{r_l}{l} \geq \frac{r_{K-1}}{K-1} \geq \frac{r_K}{K}$ yields the result. 
\end{proof}

\section*{Appendix D --- Proof of Lemma \ref{lem:mu-sym}}

Here we prove that $\mu_K^{sym} > \mu_{K+1}^{sym}$.  Specifically, 
we show that if $p$ is a value such that $0< p \leq 1$, then 
for any positive integer $K$ we have: 
\begin{equation} \label{eq:appe}
\frac{1}{K}[1 - (1-p)^K] > \frac{1}{K+1}[1 - (1-p)^{K+1}]
\end{equation} 
To show this, note that it is trivially true for the case $p=1$.  In the opposite case where
$0 < p < 1$, we can multiply (\ref{eq:appe}) by $K(K+1)$ and rearrange terms
to see that the inequality is equivalent to the following: 
\begin{equation} \label{eq:appe2} 
(1-p)^K + Kp(1-p)^K < 1
\end{equation} 
Thus, it suffices to prove that (\ref{eq:appe2}) is true.  To this end, we have: 
\begin{eqnarray*}
(1-p)^K + Kp(1-p)^K  &<& (1-p)^K + Kp(1-p)^{K-1}  \\ 
&\leq& \sum_{i=0}^K {K \choose i} p^i (1-p)^{K-i} \nonumber \\
&=& ((1-p) + p)^K = 1 \nonumber 
\end{eqnarray*}
where the first (strict) inequality holds  because $0 < p < 1$ and hence 
$(1-p)^K < (1-p)^{K-1}$. 
This establishes (\ref{eq:appe2}) and completes the proof of Lemma \ref{lem:mu-sym}.

 \section*{Appendix E --- Proof of Lemma \ref{lem:multi1}}

The queueing dynamics are given by $Q_i(t+1) = Q_i(t) -\tilde{\mu}_i(t) + A_i(t)$, 
where $\tilde{\mu}_i(t) = \min[\mu_i(t), Q_i(t)]$.  Using the Lyapunov function 
$L(\bv{Q}(t)) \defequiv \frac{1}{2}\sum_{i=1}^N Q_i(t)^2$ and performing a standard quadratic
drift computation (see, for example, \cite{now}), it is not difficult to show the 
drift satisfies: 
 \begin{eqnarray*}
\Delta(\bv{Q}(t)) &=&  \frac{1}{2}\sum_{i=1}^N \expect{A_i(t)^2}  \\
&&  -  \sum_{i=1}^N \expect{\lambda_i \tilde{\mu}_i(t) - \frac{\tilde{\mu}_i(t)^2}{2}\left|\right.\bv{Q}(t)} \\
&& + \sum_{i=1}^N \lambda_i Q_i(t) 
 -   \sum_{i=1}^N Q_i(t)\expect{\mu_i(t)\left|\right.\bv{Q}(t)}  \\
 && + \sum_{i=1}^N Q_i(t)\expect{\mu_i(t) - \tilde{\mu}_i(t)\left|\right.\bv{Q}(t)} 
 \end{eqnarray*}
  By definition of $\tilde{\mu}_i(t)$, we have: 
  \begin{eqnarray*} 
  Q_i(t)(\mu_i(t) - \tilde{\mu}_i(t)) &=& \tilde{\mu}_i(t)\mu_i(t) - \tilde{\mu}_i(t)^2 \\
  &\leq& \min[\mu_{i, max}\tilde{\mu}_i(t), \mu_i(t)^2] - \tilde{\mu}_i(t)^2
  \end{eqnarray*}   
    Hence: 
 \begin{eqnarray}
\Delta(\bv{Q}(t)) &\leq&  \frac{1}{2}\sum_{i=1}^N \expect{A_i(t)^2} \nonumber \\
&& \hspace{-.3in} -  \sum_{i=1}^N \expect{\lambda_i \tilde{\mu}_i(t) +\frac{\tilde{\mu}_i(t)^2}{2}\left|\right.\bv{Q}(t)} \nonumber \\
&& \hspace{-.3in} + \sum_{i=1}^N \lambda_i Q_i(t) 
 -   \sum_{i=1}^N Q_i(t)\expect{\mu_i(t)\left|\right.\bv{Q}(t)}  \nonumber \\
 && \hspace{-.3in}  + \sum_{i=1}^N \expect{\min\left[\mu_{i, max}\tilde{\mu}_i(t) , \mu_i(t)^2\right] \left|\right.\bv{Q}(t)}  \label{eq:appg2}  
 \end{eqnarray}
 Using the fact that $\expect{\min[\cdot, \cdot]} \leq \min[\expect{\cdot}, \expect{\cdot}]$ (by Jensen's inequality and concavity of the $\min[\cdot, \cdot]$ operator), that the sum of a $\min$ is less than or equal to the $\min$
 of a sum, and that $\sum_{i=1}^N \expect{\mu_i(t)^2\left|\right.\bv{Q}(t)} \leq \hat{S}^2$, 
 the final term on the right hand side of (\ref{eq:appg2}) 
 can be bounded by: 
 \begin{eqnarray*}
 \min\left[ \sum_{i=1}^N \mu_{i, max}\expect{\tilde{\mu}_i(t)\left|\right.\bv{Q}(t)}   ,  \hat{S}^2 \right] 
 \end{eqnarray*}
Because the max-weight policy maximizes $\sum_{i=1}^N Q_i(t)\mu_i(t)$ (given queue
backlogs $\bv{Q}(t)$)
we have: 
\begin{equation} \label{eq:app1} 
 \sum_{i=1}^NQ_i(t)\expect{\mu_i(t)\left|\right.\bv{Q}(t)} \geq \sum_{i=1}^N Q_i(t)\expect{\mu_i^*(t)\left|\right.\bv{Q}(t)} 
 \end{equation} 
where $\mu_i^*(t)$ represents any alternative scheduling decision.  Noting that $\bv{\lambda}/\rho \in \Lambda$
and $\bv{\mu}_{sym}/N \in \Lambda$, we have by convexity of $\Lambda$: 
\[ \bv{\lambda} + (1-\rho) \bv{\mu}_{sym}/N \in \Lambda \]
 Thus, there exists a stationary randomized policy that chooses $\bv{\mu}^*(t)$ independent of queue backlog
 to yield: 
 \[ \expect{\bv{\mu}^*(t) \left|\right.\bv{Q}(t)} = \expect{\bv{\mu}^*(t)} = \bv{\lambda} + (1-\rho)\bv{\mu}_{sym}/N \]
 Plugging this into (\ref{eq:app1}) and then into (\ref{eq:appg2}) yields: 
 \begin{eqnarray}
\hspace{-.2in}\Delta(\bv{Q}(t)) &\leq&  \frac{1}{2}\sum_{i=1}^N \expect{A_i(t)^2} \nonumber \\
&&  -  \sum_{i=1}^N \expect{\lambda_i \tilde{\mu}_i(t) +\frac{\tilde{\mu}_i(t)^2}{2}\left|\right.\bv{Q}(t)} \nonumber \\
&& -  \frac{(1-\rho)\mu_{sym}}{N} \sum_{i=1}^N Q_i(t) \nonumber \\
 && + \min\left[\sum_{i=1}^N \mu_{i, max}\expect{\tilde{\mu}_i(t) \left|\right.\bv{Q}(t)}, \hat{S}^2\right] \label{eq:appg3}  
 \end{eqnarray}
 Using the Lyapunov drift lemma (Lemma \ref{lem:lyap-drift}) on the above drift and noting that the 
 system is stable with well defined time average limits yields: 
 \begin{eqnarray*} 
 \sum_{i=1}^N \overline{Q}_i &\leq&   \frac{N\left[\frac{1}{2}\sum_{i=1}^N \expect{A_i^2} - \frac{3}{2}\sum_{i=1}^N \lambda_i^2  \right]}{(1-\rho)\mu_{sym}} \\
 && + \frac{N\min\left[ \sum_{i=1}^N \lambda_i\mu_{i,max}, \hat{S}^2 \right]}{(1-\rho)\mu_{sym}}
 \end{eqnarray*}
 where we have used the fact that $\lim_{t\rightarrow\infty}\expect{\tilde{\mu}_i(t)} = \lambda_i$ and 
 $\lim_{t\rightarrow\infty} \expect{\tilde{\mu}_i^2(t)} \geq \lim_{t\rightarrow\infty}\expect{\tilde{\mu}_i(t)}^2 = \lambda_i^2$.  Using Little's theorem on this
 congestion bound proves (\ref{eq:original-multi}). 
 
 Now suppose that 
 all $S_i(t)$ processes are independent and $Pr[S_i(t)  = \mu_{i,max}] \geq p_{min}$ for all $i$. 
 We derive the bound on $\mu_{sym}$ given in Lemma \ref{lem:multi1}. Define
 $\hat{\mu} \defequiv \min_{i\in\{1, \ldots, N\}} \mu_{i,max}$. Consider the stationary 
 and randomized algorithm $\bv{\mu}^*(t)$  that observes channel states $\bv{S}(t)$ and probabilistically places each 
 link $i \in \{1, \ldots, N\}$ in a set $\chi(t)$ with probability $0$ if $S_i(t) <  \hat{\mu}$, and 
 with probability $p_{min}/Pr[S_i(t) \geq \hat{\mu}]$ if $S_i(t)  \geq \hat{\mu}$.  Then $\chi(t)$ contains
 a random number of links, and each link appears in $\chi(t)$ independently  
 with probability $p_{min}$.  Select a link to serve
 on slot $t$ uniformly and randomly with equal probability over all links in $\chi(t)$ (remaining idle if 
 $\chi(t)$ is empty).  It follows that under this policy, a particular link $i$ is selected for transmission
 with probability  exactly 
 $(1 - (1-p_{min})^N)/N$, and is selected only if $S_i(t) \geq \hat{\mu}$.  Hence: 
 \[ \expect{\mu_i^*(t)} \geq \frac{\hat{\mu}(1-(1-p_{min})^N)}{N} \: \: \mbox{ for all $i \in \{1, \ldots, N\}$}  \]
 It follows that the symmetric rate vector with all $N$ entries equal to the right hand side in the above
 expression is in the capacity region $\Lambda$, so that $\mu_{sym}/N$ is greater than or equal to this value. 
 
\bibliographystyle{unsrt}
\bibliography{../../../../../latex-mit/bibliography/refs}

\end{document}